\documentclass[11pt]{article}
\usepackage{amsfonts,amssymb,amsmath,amsthm}
\usepackage{graphicx,caption}
\usepackage[all]{xy}
\usepackage{amscd,color}
\usepackage[shortlabels]{enumitem}
\usepackage[normalem]{ulem}

\newtheorem{theorem}{Theorem}[section]
\newtheorem{corollary}[theorem]{Corollary}
\newtheorem{lemma}[theorem]{Lemma}
\newtheorem{proposition}[theorem]{Proposition}
\newtheorem{conjecture}[theorem]{Conjecture}
\newtheorem*{thmA}{Theorem A}

\newtheorem*{thmC}{Theorem C}
\newtheorem*{thmD}{Theorem D}

\newtheorem*{propB}{Proposition B}

\theoremstyle{definition}
\newtheorem{definition}[theorem]{Definition}

\newtheorem{example1}{Example}

\theoremstyle{remark}
\newtheorem{remark}[theorem]{Remark}

\newtheorem*{remark*}{Remark}

\numberwithin{equation}{section}

 \divide\oddsidemargin by 2
\newcommand{\cala}{{\cal A}}
\newcommand{\calb}{{\cal B}}

\newcommand{\calf}{{\cal F}}

\newcommand{\calh}{{\cal H}}

\newcommand{\calm}{{\cal M}}

\newcommand{\cals}{{\cal S}}

\newcommand{\C}{{\mathbb C}}
\newcommand{\D}{{\mathbb D}}
\renewcommand{\H}{{\mathbb H}}

\newcommand{\N}{{\mathbb N}}

\newcommand{\R}{{\mathbb R}}

\newcommand{\Z}{{\mathbb Z}}
\newcommand{\chat}{{\widehat{\mathbb C}}}

\renewcommand{\hat}{\widehat}

\newcommand{\la}{\lambda}
\newcommand{\las}{\lambda^*}
\newcommand{\inv}{^{-1}}
\newcommand{\lo}{{\la_0}}

\renewcommand{\r}{\color{red}}

\newcommand{\pmf}{\calm_\infty}

\newcommand{\ps}{\Lambda}
\newcommand{\lap}{{\la'}}
\newcommand{\bA}{{\bf A}}
\newcommand{\ba}{{\bf a}}

\newcommand{\lp}{\H_l}

 \title{Stable components in the parameter plane of transcendental functions of finite type}
\author{\small N\'uria Fagella\thanks{Partially supported by spanish grants MTM2017-86795-C3-3-P and MDM-2014-0445 Mar\'ia de Maeztu and the catalan grant 2017 SGR 1374. }\\  
\small Dept.~de Matem\`atiques i Inform\`atica\\ 
\small Univ.~de Barcelona, Gran Via 585 \\ 
\small Barcelona Graduate School of Mathematics (BGSMath)\\
\small 08007 Barcelona, Spain\\  
\small {\tt fagella$@$maia.ub.es} 
\and 
\small Linda Keen\thanks{Partially supported by   PSC-CUNY grants 67260-00 45, 80209-05 20  }\\   
\small CUNY Graduate Center\\ 
\small  365 Fifth Avenue, New York \\ 
\small NY, NY 10016 \\  
\small {\tt LKeen$@$gc.cuny.edu, Linda$@$keenbrezin.us} 
} 

\begin{document} 

\maketitle  
\begin{abstract}
 We study the parameter planes of certain one-dimensional, dynamically-defined slices of holomorphic families of entire and meromorphic transcendental maps of finite type.  Our planes are defined by constraining the orbits of all but one of the singular values, and leaving free one asymptotic value. We study the structure of the regions of parameters, which we call {\em shell components},  for which the free asymptotic value tends to an attracting cycle of non-constant multiplier.   The exponential and the tangent families are examples that have been studied in detail, and the hyperbolic components in those parameter planes are shell components. Our results apply to slices of both entire and meromorphic maps. We prove that shell components are simply connected,  have a locally connected boundary and have no center, i.e., no parameter value for which the cycle is superattracting. Instead, there is a unique parameter in the boundary, the {\em virtual center}, which plays the same role. For entire slices, the virtual center is always at infinity, while for meromorphic ones it maybe finite or infinite. 
In the dynamical plane we prove, among other results, that the basins of attraction which contain only one asymptotic value and no critical points are simply connected. Our dynamical plane results apply without the restriction of finite type.
 \end{abstract}

\section{Introduction}
 Our starting point is a holomorphic family $\calf=\{f(x,z)=f_x(z) : X\times \C \rightarrow \hat\C \}$ of transcendental entire or meromorphic functions defined over a complex analytic manifold, the parameter space.  By a meromorphic map we mean a map with at least one pole which is not omitted. 
    Iterating $f_x$ for a given value of $x$  gives rise to a dynamical system.  A fundamental problem is to understand what bifurcations occur in the dynamics as one varies the parameter $x$.  
 In this generality, the problem is very difficult and no results exist.  Because the dynamics of these systems are controlled by the orbits of the singular values of the maps $f_x$, one effective way to approach the problem is to define one-dimensional subspaces, or ``slices'' through the manifold $X$  that constrain all but one of these singular values and analyze the structure within these slices. 
 For the family of quadratic polynomials this analysis has been carried out more or less completely. (See for example \cite{orsaynotes1,orsaynotes2,cargam}).  For polynomials of higher degree and rational maps, this is also true to a lesser extent (see e.g. \cite{rees95,BH1,BH2,bert}),  and this analysis  shows that behavior for quadratics is not special, but is a paradigm for rational dynamics as well.  To wit, one often sees copies of Mandelbrot sets in one dimensional slices of rational or transcendental families; the slice of degree two rational maps  with an attracting fixed point and fixed multiplier, or Newton's method for cubic polynomials are examples of this \cite{douhub, gold-keen2}).

The exponential  and the tangent family were the first examples of transcendental entire and meromorphic functions respectively  to be systematically studied (e.g. \cite{bakrip,rempethesis,devfagjar,KK}). The essential new feature for transcendental functions is that they have asymptotic values, which are, in effect, images of infinity along certain paths.  This introduces new phenomena into the structure of the parameter spaces that are not seen for rational maps.  In the same way that the Mandelbrot-like sets code the influence of the critical values of the dynamics of rational maps, these new phenomena code the influence of the asymptotic values on the dynamics of our families of transcendental maps.

This work is the first to find a context for studying  parameter spaces of  large classes of both transcendental entire and meromorphic maps.    Our goal is to show that in holomorphic families of such functions, there are natural one-dimensional slices in the parameter space  in which there are always certain kinds of components whose properties are characterized by the existence of asymptotic values.

We define a one-dimensional  slice by allowing only one ``free" asymptotic value to vary.  
Our main results 
concern properties of the parameter regions for which this free asymptotic value is attracted to an attracting cycle.  These regions, which we  call ``shell components", are the analogues in this setting of the hyperbolic components of the interior of the Mandelbrot-like sets for rational maps. 

To make these ideas precise,  consider the dynamical system formed by iterating an entire or  meromorphic transcendental map of the complex plane 
$$
f:\C \to \C\cup \{\infty\}.
$$
Because of  the essential singularity at infinity, the global degree of $f$ is infinite and the dynamics is richer than for rational maps. The exponential and the tangent maps are in this class.  Other examples are generated by applying  Newton's method   to an entire function, which almost always yields a transcendental meromorphic map as the root finder function to be iterated.

The dynamical plane of $f$ splits into two disjoint completely invariant subsets: the {\em Fatou set}, or points for which the family of iterates is well defined and form a normal family in some neighborhood; and its complement, the {\em Julia set}. The Fatou set is open and contains, among other types of components,  all basins of attraction of attracting periodic orbits. By contrast, the Julia set is closed in the sphere and it is either the whole sphere or has no interior.  It  can be characterized as the closure of the repelling periodic points or, if the map is meromorphic, as the closure of the set of poles and prepoles. For general background on meromorphic dynamics we refer to the survey \cite{bergweiler} and the references there.  

A key role in the dynamics of a transcendental map $f$ is played by points in the set $S(f)$ of  {\em singular values} of $f$; these are  either {\em critical values} (images of  zeroes of $f'$, the {\em critical points}), or {\em asymptotic values} (points $v = \lim_{t \to \infty} f(\omega(t))$ where $\omega(t) \to \infty$ as $t\to\infty$),  or accumulations thereof.  Indeed every periodic connected component of the Fatou set  is, in some sense, associated to the orbit of a point in $S(f)$.   For example,  basins of attraction of an attracting or parabolic cycle   must actually contain a singular value, but weaker conditions associating a singular value to  Siegel disks or Baker domains \cite{milnor,ber95,benfag2} and even  to wandering domains \cite{BFJKsing} also exist. Note that $S(f)$ coincides with the set of singularities of the inverse map.

Transcendental entire or meromorphic functions with a finite number of singular values have finite dimensional parameter spaces \cite{EL92,gold-keen}. These are  known as {\em finite type} maps and denoted by 
\[
\cals :=\{f:\C \to \C\cup \{\infty\} \text{\ entire or meromorphic}  \mid \#{S(f)} <\infty  \}.
\]
Maps in the class $\cals$ have several special properties such as the absence of wandering domains \cite{sullivan,gold-keen,EL92,bkl4} or the existence of at most a finite number of non-repelling cycles.

Entire transcendental functions,  whose dynamics have attracted a great deal of attention in the last few years, may  be thought of as  special cases of transcendental meromorphic maps in a wider sense, namely those with all their poles at infinity.   The simplest example is the exponential family $E_\la(z)=e^z+\lambda$.    By Iversen's theorem \cite{ivethesis} entire functions (and meromorphic functions with finitely many poles),  all have infinity as an asymptotic value, which may be taken to mean that infinity is a multiple pole.  
By contrast, meromorphic maps for which infinity is not an asymptotic value are at the other end of the spectrum:  they  necessarily have infinitely many poles and  this difference in behavior at infinity has consequences for the dynamical and parameter spaces (see e.g. \cite{berkot}).  The tangent family of maps $T_\lambda(z)=\lambda \tan(z)$ \cite{DK, KK, GaKo} is the simplest example of these. 
 This class, which we denote as 
 \[
\pmf :=\{ f \in \cals \mid \text{$\infty$ is not an asymptotic value}\},
\] 
in some sense captures the essence of meromorphic maps  since   the condition that infinity {\em is not}   an asymptotic value persists under small perturbations  of $f$ in $\cals$.

To center the discussion, we consider {\em dynamically natural families  (or slices)} of functions in $\cals$ which, roughly speaking, consist of one-dimensional slices of larger holomorphic families of entire or meromorphic maps, for which only one singular value is allowed to bifurcate at any given parameter value  (see Definition \ref{def:dns} for details). We concentrate our attention on the behavior of a marked asymptotic value called the {\em free asymptotic value}, and in particular, on the connected components of parameters for which this marked point is attracted to an attracting cycle whose basin contains no other singular value. We call these sets {\em shell components} and  they are the transcendental version of the hyperbolic components one finds in parameter planes of rational maps.  Note, however,  that  because  we allow  parabolic cycles to exist and attract a non-free singular value, the maps here need not be hyperbolic.  While shell components occur in the parameter planes for many families of entire functions, they are better understood in the broader context which includes meromorphic functions.  

The first part of the paper  deals with the properties of attracting basins in the dynamical plane of maps belonging to a shell component.  Our main results about the dynamical plane are  proved  in Section \ref{dynamical plane},  and are summarized (in a slightly weaker form) in the following statement. 
\begin{thmA}
Let $f$ be an entire or meromorphic function. Suppose that $a_0,\ldots, a_{p-1}$ is an attracting cycle of period $p\geq 1$ whose basin of attraction $\cala$ contains exactly one asymptotic value $v$ and no critical points. For $i=0,\ldots,p-1$ let $A_i$ be the component of $\cala$ containing $a_i$ and labeled so that $f(A_i)=A_{i+1 \pmod p}$ and $v\in A_1$. Set $A:=A_0\cup \cdots \cup A_{p-1}$, the immediate basin of attraction. Then:
\begin{enumerate}[\rm (a)]
\item Every component of $\cala$ is simply connected;
\item  $A$ contains no finite preimage of $v$;
\item $A_0$ is unbounded and maps infinite to one onto $A_1\setminus \{v\}$ while $f:A_j\to A_{j+1}$ is univalent for all $j\neq 0$;
\item If in addition,  $S(f)\cap \partial A=\emptyset $, then $A$ contains only one asymptotic tract of $v$.
\end{enumerate}
\end{thmA}

 Note that the only assumption on $f$ in Theorem A is that it is transcendental. In particular,  it may have infinitely many singular values.

The rest of the paper concentrates on properties of parameter spaces. A dynamically natural slice is represented by a one-dimensional manifold,  $\Lambda \subset X$, isomorphic to the complex plane with  perhaps a discrete set of punctures called {\em parameter singularities}, where,  for example, the function may reduce to a constant.
For simplicity   we require the parametrization to trace the position of the marked asymptotic value in an affine fashion.
Most of the best known one-parameter slices such as  quadratic rational maps with a super attracting cycle (quadratic polynomials),  cubic polynomials with a super attracting fixed point, exponential or tangent functions, and so on,  are dynamically natural slices of maps in different classes  where the parameter is an affine function of a singular value  (see Section \ref{sec:examples2}).  

If $\Lambda$ is a dynamically natural slice for the family $\calf$ of entire or meromorphic maps  in $\cals$,  we call the restriction of the family to $\Lambda$  a ``dynamically natural family'' and denote it by $\{f_\la\}_{\la\in\Lambda}$.  The slice  $\Lambda$ contains  different types of ``distinguished" parameter values which are solutions of algebraic or transcendental equations. Slightly abusing the usual terminology we call those parameter values for which  the free  singular value lands on a repelling cycle   {\em Misiurewicz parameters}.  {\em Centers}, parameters for which a critical point is periodic, are also distinguished parameters.    
Meromorphic transcendental maps have an  additional  type of distinguished parameter which does not occur for rational maps.  We call a parameter $\la$ a {\em virtual cycle parameter} if some iterate of the asymptotic value is a pole for $f_\la$. This name is motivated by the existence of  a {\em virtual cycle},  morally a ``cycle" that contains the point at infinity and the asymptotic value (see Definition \ref{def:vc}). Virtual cycle parameters are important in our discussion and they are generally quite abundant in the following sense (see Proposition \ref{dense}). 

\begin{propB}
Let   $\{f_\lambda \}_{\la \in \Lambda}$ be a dynamically natural family of entire or meromorphic  maps in $\cals$, and let $v=v(\la)$ be the free asymptotic value. Both the virtual cycle parameters, when they exist,    and the  Misiurewicz parameters   are dense in  the set 
\[
\calb(v) := \{ \lambda_0 \in \Lambda \mid \{\la \mapsto f_\la^n(v_\la)\}  \text{ is not a normal family in any neighborhood of $\la_0$}\}. 
\]
Moreover  $\calb(v)$ has no isolated points in $\ps$.
\end{propB}

The set $\calb(v) $ is known as the {\em activity locus} of the free asymptotic value and is part of the {\em bifurcation locus} $\calb$, or the set of parameters for which at least one of the singular values  bifurcates. It follows from our definition of a dynamically natural slice, that  the boundary of any shell component is a subset of $\calb(v)$. 

 General results of this nature have been shown for  parameters spaces of rational maps using bifurcation currents  (see for example \cite{bert,dujfav,duj,gaut}).

By definition, functions in a shell component of a dynamically natural slice have no critical points in the basin of the attracting cycle,  so the  multiplier of the cycle is never zero, and therefore the component has no {\em center}. We see however, that some virtual cycle parameters play the role of centers in our setting (see Theorem \ref{vc1}). 

\begin{definition}[Virtual center] 
Let $\Omega$ be a shell component of a dynamically natural  family $ \{f_\lambda \}_{\la \in \Lambda}$ of entire or meromorphic maps. We say that $\la^* \in (\partial \Omega \cap \Lambda)\cup \{\infty\}$ is a {\em virtual center} of $\Omega$, if there exists a sequence of $\la_n\in\Omega$ such that $\la_n\to\la^*$ and the multiplier $\rho_\Omega(\la_n)$ of the attracting cycle for $\la_n$ tends to $0$.
\end{definition}

 Here and throughout the paper,  when we say a point is  {\em finite} we mean that it  is not the point at infinity.
\begin{thmC}
Let   $\{f_\lambda \}_{\la \in \Lambda}$ be a dynamically natural family of entire or meromorphic  maps in $\cals$ and let $\Omega$ be a shell component. Then, if  $\la^* \in \partial \Omega \cap \Lambda$ is a virtual center,  $\la^*$ is a virtual cycle parameter.  Thus, if the functions are entire,  there are no finite virtual centers.
\end{thmC}

Although maps in a shell component might be non-hyperbolic, a shell component $\Omega$ shares many properties with hyperbolic components of known families like the exponential or tangent families, \cite{rempethesis,KK}. Indeed,  shell components can be parametrized  by the multiplier map
\[
\rho=\rho_\Omega : \Omega \longrightarrow \D^*
\]
where $\D^*:=\D\setminus \{0\}$, and $\rho(\lambda)$ is the the multiplier of the attracting cycle to which $v(\lambda)$ is attracted. Since the multiplier is never zero, $\rho$ is a covering (see Theorem~\ref{multcov}) and hence shell components are either simply connected and $\rho$ has infinite degree or else they are  homeomorphic to $\D^*$, where the puncture is a parameter singularity (see Theorem \ref{coverings} and Corollary \ref{dichotomy}).  

Our last result describes the internal structure and the boundary of a shell component and establishes the existence and uniqueness of the virtual center (see Theorem \ref{boundarycont},  and corollaries \ref{cor:uniquevc}, \ref{cor:uniquevc}  and \ref{cor:internalrays}). 

\begin{thmD}
Let   $\{f_\lambda \}_{\la \in \Lambda}$ be a dynamically natural family of entire or meromorphic  maps in $\cals$  and  let $\Omega$ be a simply connected shell component. Let $p\geq 1$ be the period of the attracting cycle throughout $\Omega$. Then:
\begin{enumerate}[\rm (a)]
\item $\partial \Omega$ is locally connected and hence a continuous curve.
\item $\partial \Omega$ has a unique virtual center.
\item If $p=1$  or if the family is entire, the virtual center is at infinity and therefore $\Omega$ is unbounded. 
\item Define the internal ray in $\Omega$ of angle $\theta$ as 
\[
 R_\Omega(\theta):=\{  \rho^{-1}_{\Omega}(\exp(t+2\pi i \theta)),   \,  t\in (-\infty,0)\}.
\]
Then, all internal rays have one end at the virtual center and the other end at a point in $\partial \Omega$ (which a priori could be infinity). If the virtual center $\la^*$  is finite  no internal ray has both ends at $\la^*$. 
\end{enumerate}
\end{thmD}

When the shell component is not simply connected then the puncture, which is a parameter singularity, plays the role of the virtual center. An example of a doubly connected shell component can be seen in Example 1(d) in Section \ref{sec:examples2}. 

\begin{figure}[htb!]
\captionsetup{width=0.85\textwidth}
\begin{center}
\includegraphics[height=4cm]{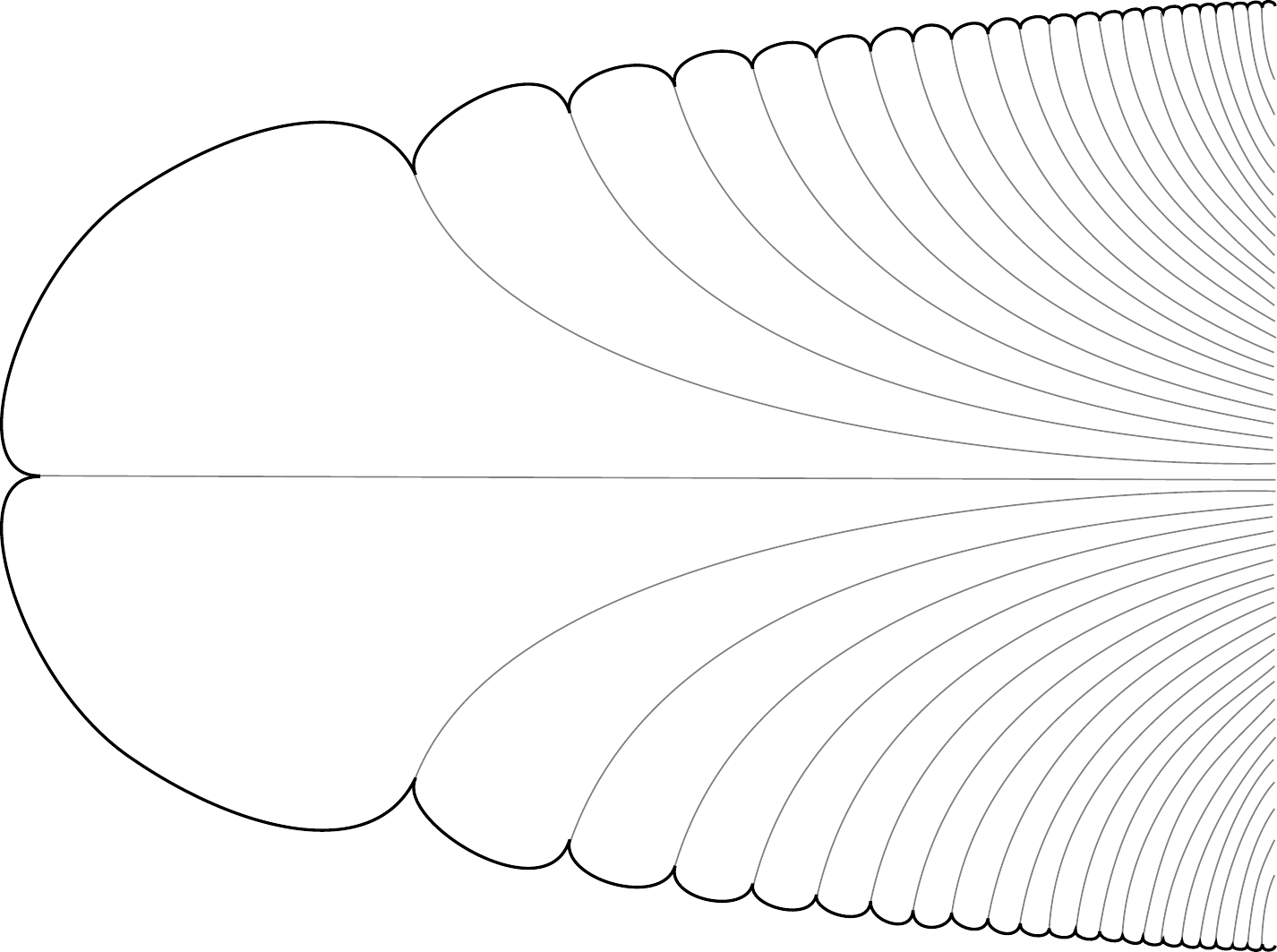} \hfil
\includegraphics[height=4cm]{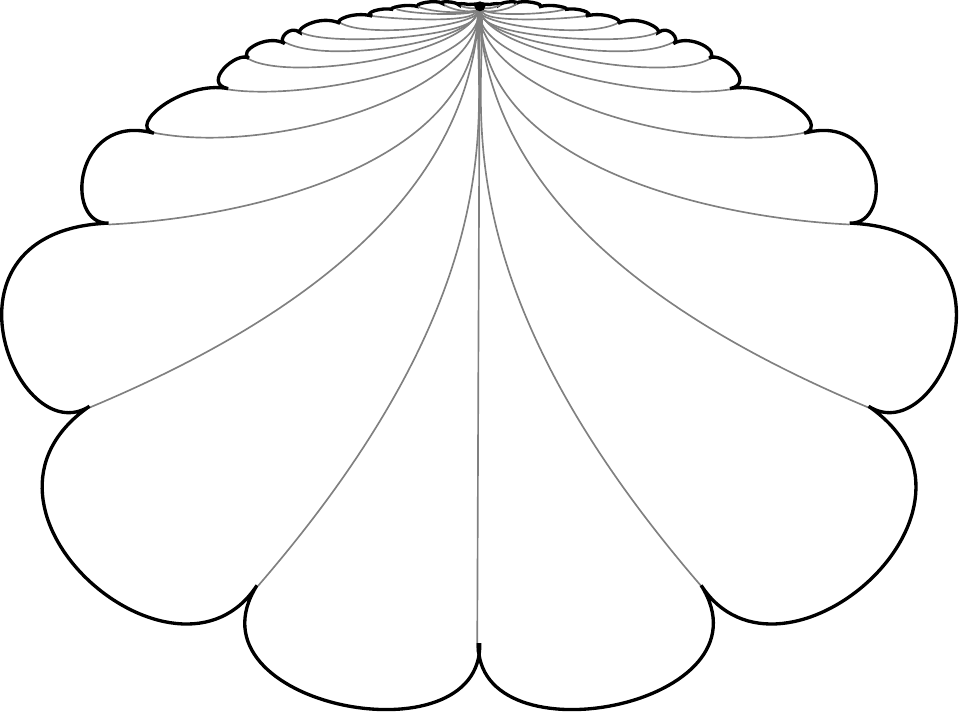}
\end{center}
\caption{\label{sketch}
 \small Sketch of two shell components and their internal rays. On the left (resp. right) the virtual center is at infinity (resp. at a finite point) and therefore the component is unbounded (resp. bounded).}
\end{figure}

Numerical experiments show that for meromorphic functions in $\pmf$  shell components of period larger than one are always bounded, that is, they have their virtual center at a finite parameter value, so that it seems reasonable to conjecture that this is so in general. At present,  we can only prove this fact for particular families like the tangent family. (See \cite{CK19}).   

The paper is structured as follows. In Section~\ref{sings} we recall the definition of singular values of transcendental functions and we state some standard theorems on the covering  and connectivity properties of their Fatou components.  In Section~\ref{dynamical plane} we discuss the properties of the dynamical planes of  these functions, proving Theorem A. In Section~\ref{slices} we define the {\em dynamically natural slices} of parameter spaces that are the main subject of the paper and in Section~\ref{sec:examples2} we give a number of examples.    In Section~\ref{Dist values} we define various types of distinguished parameter values that lie in the bifurcation locus and prove Proposition B.   Section~\ref{Shell} is the heart of the paper where, after some preliminary results, we prove Theorems C and D.  We conclude with an Appendix where we discuss a theorem of Nevanlinna that allows us to characterize certain functions of finite type.  We also extend this theorem to a somewhat larger class of functions.   

\noindent {\bf Acknowledgements} \ The authors would like to thank the referee for a very careful reading and helpful comments which have not only improved the exposition, but also the results.   We are also grateful to  CUNY Graduate Center and to IMUB at Universitat de Barcelona for their hospitality while this paper was in progress.

\section{Preliminaries and Setup}\label{sings}

\subsection{Singularities of the inverse function}  \label{subs:sings}

Let $f$ be a transcendental entire or meromorphic map. A point $v\in\C$ is a {\em singular value} of $f$ if some branch of $f\inv$  fails to be well defined in every small enough neighborhood of $v$. Singular values may be critical, asymptotic or accumulations thereof.   If $c$ is a  critical point,  that is,  a zero of $f'$  then its image $v=f(c)$ is a  {\em critical value}.  If there is a path $\omega(t)$ such that $\lim_{t \to 1} \omega(t)=\infty$ and $\lim_{t \to 1} f(\omega(t))=v$ then the limit $v$ is an {\em asymptotic value}  of $f$. Observe that the point at infinity is always an essential singularity, but it may be or not be a singular value. It is a critical value if and only if it is the image of a multiple pole of $f$ and it is an asymptotic value if some unbounded curve has an image that tends to infinity. For example, infinity is not a singular value for the tangent map.    

Singular values are classified in terms of the branches of $f^{-1}$ as follows (see \cite{berere},\cite{ive}).
 
 \begin{proposition}[Classification of singularities]\label{Singularities}
Let $f$ be an entire or meromorphic transcendental map.  For any $z\in\C$ and   $r >0$, let  $B(z,r)$ be the disk of radius $r$ centered at $z$.    Let $U_r$ be a connected component of $f^{-1}(B(z,r))$, chosen such that $U_r\subset U_{r'}$ if $r<r'$. Then there are only two possibities:
\begin{itemize}\item[(a)] $\bigcap_r U_r=\{p\}, p\in\C$
\item[(b)] $\bigcap_r U_r=\emptyset$.
\end{itemize}  
In case $(a)$,  $f(p)=z $, and either $f'(p)\neq 0$ and $z$ is a regular point, or $f'(p)=0$  and $z$ is a critical value.
In case $(b)$,   the chosen inverse branch with image $U_r$ defines a transcendental singularity over $z$, and it can be shown that  $z$ is an asymptotic value for $f$.  
\end{proposition}

The map $f$ is said to be of {\em finite type} if it has a finite number of singular values.   Note that this implies all the singular values are isolated. 

If an asymptotic value $v$ is isolated,  the radius $r$ in the above proposition 
can be chosen small enough so that $f:U_r \rightarrow B(v,r) \setminus\{v\}$ is a universal covering map.  In this case $U_r$ is called an {\em asymptotic tract} for the asymptotic value $v$ and $v$ is called a {\em logarithmic singular value} or a {\em logarithmic singularity}.  The number of distinct asymptotic tracts of a given asymptotic value is called its {\em multiplicity}.  Note that this notion is not the same as the multiplicity of a critical value,   that is, the sum  of the multiplicities (local degree minus one) of $f$ over all its preimages.   Using this term in these two ways should not cause confusion.

We denote the set of singular values of $f$ by $S(f)$ and the post-singular set by
\[ 
P(f) = \overline{ \bigcup_{x \in S_f}  \bigcup_{n \geq 0} f^n(x)}. 
\]
%
By Iversen's Theorem \cite{ivethesis}  an entire transcendental function $f$ always has an asymptotic value at infinity.

\subsection{Mapping properties}

The following lemma is well known in algebraic topology. See for example \cite{Massey} for the general theory of coverings or \cite[Thm. 6.1.1.]{zheng10} for a proof of the lemma below.

Let $U,V\subset \C$. Recall that a map $f:U\to V$ is a {\em covering map}  if for every $z\in V$ and every small enough  neighborhood $N_z$ of $z$,  $f^{-1}(N_z)$ consists of a collection of disjoint topological disks $\{D_i\}_i$ such that $f:D_i\to N_z$ is a homeomorphism.  If $f:D_i\to N_z$ is equivalent to $z^d$ for some $d>1$, $f$ is a {\em branched covering of order $d-1$}. If $f$ is unbranched and $U$ is simply connected we say that $f$ is a {\em universal covering}.


\begin{lemma}[Holomorphic coverings of $\D$ and $\D^*$] \label{coverings}
Let $U\subset \chat$ be an open set, $\D$ be the unit disk and and $D^*=\D\setminus \{0\}$. 
\begin{enumerate}[\rm(a)]
\item If $f:U\to \D $ is a holomorphic  covering map, then $U$ is simply connected and $f$ is univalent. 
\item If $f:U\to \D^*$ is a holomorphic   covering map, then either 
\begin{itemize}
\item[(i)]  $U$ is conformally equivalent to $\D^*$ and there exists a biholomorphic mapping $\psi:U\to \D^*$,  such that 
$f=(\psi)^d$ for some $d\in \N$, or
\item[(ii)]  $U$ is simply connected and there exists a biholomorphic mapping $\psi: U \to \lp = \{z : {\rm Re} z<0 \}$ such that $f =\exp\circ \psi$.
\end{itemize}
\end{enumerate}
\end{lemma}

\section{Dynamical plane. Proof of Theorem A.}\label{dynamical plane}

 In this section we describe some properties of the dynamical plane  for  general entire or meromorphic maps. In particular, we do not assume our maps to have a finite number of singular values.  Theorem A in the introduction follows from the three propositions in this section, which are slightly stronger than the theorem itself. 

We are interested in basins of attraction of attracting cycles containing a unique asymptotic value $v$.  Since $v$ is the only singular value in the basin, it is isolated and is thus  a logarithmic singularity (see Section \ref{subs:sings}).   It follows that  if $V^*$ is a punctured neighborhood of $v$, there exists at least one asymptotic tract  among the components of $f^{-1}(V^*)$.   
  
We shall prove that under this weak assumption, a great deal can be said about the attracting basin. We begin with connectivity properties.

\begin{proposition}[Simply connected basin]\label{sc}
Let $f$ be an entire or meromorphic transcendental function.  Suppose that  $a_0, \cdots ,a_{p-1}$ is an attracting  cycle of period $p\geq 1$ whose  immediate basin of attraction $\bA$ contains exactly one asymptotic  value $v$ and no critical points.  Then every component of  $\bA$ is simply connected and $\bA$ contains no finite preimage of $v$.  If, moreover, there are no other singular values or critical points in the whole basin of attraction $\cala$, then every component of $\cala$ is simply connected.
\end{proposition}

\begin{proof}
For $k=0, \ldots, p-1$, let $A_k$ denote the component of $\bA$ which contains $a_k$ where labels are defined so that  $f(A_k)=A_{k+1\pmod p}$. Let $U_0\subset A_0$ be a topological open disk with a smooth boundary containing $a_0$ in its interior and such that $f^p$ is one to one in $U_0$ and $f^p(\overline{U_0}) \subset U_0$. This set exists by Koenig's Linearization Theorem (see e.g. \cite{milnor}). Consider the successive preimages of $U_0$ which follow the cycle in reverse order; that is for every $k\geq 1$, recursively define the set $U_k \subset A_{(p-k)\bmod p}$ as the component of $f^{-1} (U_{k-1})$ that contains the point $a_{(p-k)\bmod p}$. 
This defines a nested sequence of open sets around each of the points in the cycle. That is, for every $0\leq k \leq p-1$ we have 
\[
U_k \subset U_{k+p} \subset U_{k+2p}\cdots
\]
and the infinite union of these nested sets equals $A_k$. 
By definition, since every component $A_k$ is the infinite union of a collection of nested open sets, 
  the components of the immediate basin are simply connected if and only if $U_k$ is simply connected for every $k$. 

We will use induction to show every $U_k$ is simply connected.   For $k=0$,  this is true by construction. We assume now that $U_{k}$ is simply connected and    prove that the same is true for $U_{k+1}$.

First suppose that $v\notin U_{k}$. Then $U_{k}$ has no singular value and so $U_{k+1}$ has no critical point and thus  $f:U_{k+1} \to U_{k}$ is a  covering map. Since $U_k$ is conformally equivalent to a disk by hypothesis, it follows from Lemma~\ref{coverings}   that $U_{k+1}$ is simply connected and $f$ is conformal on this set. 
 
Now suppose that $v\in U_{k}$.  Then $U_{k+1}$ either contains an asymptotic tract of $v$ and $f|_{U_{k+1}}$ has infinite degree, or it does not. Suppose it does. By hypothesis $f:U_{k+1}\setminus \{f^{-1}(v)\} \to U_k\setminus \{v\}$ is a covering map.   Since $U_k\setminus \{v\}$ is conformally equivalent to a punctured disk,  it follows from Lemma~\ref{coverings} that $f^{-1}(v) \cap U_{k+1}$  contains at most one point. But if it contained a point, $f|_{U_{k+1}}$ would have finite degree, and that is not the case. Hence $ f^{-1}(v) \cap U_{k+1} =\emptyset $ and $U_{k+1}$ is simply connected and we are done. 

Now suppose  $U_{k+1}$ does not contain an asymptotic tract. In this case,  $f:U_{k+1}\setminus \{f^{-1}(v)\} \to U_k\setminus \{v\}$ is still a covering, but now  of finite degree, and  thus $f^{-1}(v)$ is a single point, a finite preimage $v'$ of $v$.  Note that this degree must be one because there are no critical values in the basin.   Eventually,  however, $U_{k+qp+1}$ must contain an asymptotic tract of $v$ for some $q>0$ because the immediate basin must contain a singular value  and $v$  is the only such by hypothesis.   Since the sets are nested, $U_{k+qp+1}$ would then contain not only $v'$ but also the asymptotic tract of $v$ so that  $f:U_{k+1+qp} \to U_{k+qp}$ has infinite degree and would contradict part (ii) of Lemma~\ref{coverings}.

We conclude that every component of the immediate basin is simply connected  and $v$ has no finite preimages in $\bA$. 

If furthermore $\cala$ contains no critical points, the same arguments applied to preimages of every $A_k$ for every $k$, show that at every component $\mathcal{A}$ is simply connected.  
\end{proof}

\begin{proposition}[Simple asymptotic value] \label{mapprops}
In the setup of Proposition \ref{sc}, suppose that $p\geq 2$ and let $A_0,\ldots,A_{p-1}$ be the components of the immediate basin  of attraction $\bA$ of the attracting $p-$cycle, indexed so that $v\in A_1$, $f(A_j)=A_{j+1\pmod p}$, and $a_j\in A_j$ for $0\leq j < p-1$. Let $S(f)$ be the set of singular values of $f$ and suppose $S(f)\cap \partial \bA=\emptyset$. Then,  
\begin{enumerate}[{\rm (a)}]
\item $A_0$ is unbounded and maps infinite to one onto $A_1 \setminus \{v\} $. Moreover infinity is accessible from $A_0$. 
\item  $f:A_j \to A_{j+1}$ is one to one for all $j \neq 0 \, {\rm mod} \,p$.  If  $f\in\calm_\infty$, then  $\partial A_{p-1}$ contains an accessible pole.
\item $A_0$ contains only one asymptotic tract of $v$ (i.e., $v$ has multiplicty one in $\bA$)
\end{enumerate}
\end{proposition}

\begin{proof}
Let $B\in A_1$ be a neighborhood of  $v$ with a smooth boundary. Then $f^{-1}(B)$ contains a simply connected unbounded set $T$ in $A_0$ (an asymptotic tract) whose boundary is a open curve which tends to infinity in both directions, showing that infinity is accessible from $A_1$. Moreover, $f:T\to B$ has infinite degree, hence $f:A_0\to A_1$ has infinite degree.   This shows (a).

Since $v \in A_1$ is the only singular value in the basin, $f:A_j\to A_{j+1 \pmod p}$ is a  covering map for all $j\neq 0$. Consequently, since all components of the basin are simply connected, these maps are all univalent.

  Let $\gamma:[0,1)\to A_0$ be any curve tending to infinity as $t\to 1$, and let $\sigma:[0,1) \to A_{p-1}$  be a path satisfying   $ \gamma= f \circ \sigma$. By continuity, any accumulation point of $\sigma(t)$ when $t\to 1$ must be infinity or a pole. Since the  set of poles is discrete, it follows that $\sigma$ must actually land at infinity or at a pole of $f$. 
   If  $\sigma$  were to land on infinity,   infinity would be an asymptotic value.  Therefore, if we  assume that $\infty$ is not an asymptotic value of $f$, it follows that it lands at a  pole in $\partial A_{p-1}$;  it is accessible  from $A_{p-1}$ (i.e., the limit of an arc $\sigma(t) \subset A_{p-1}$) by construction. This ends the proof of (b).

 The strategy for proving (c) relies on showing  that if $A_0$ contains two asymptotic tracts of $v$, it must also contain a critical point of $f$,  contradicting the hypotheses.   Although this is probably well known, we did not find a reference for it and therefore we include a proof. We need to set up some notation. Let  $\delta:[0,1]\to \C$ be a simple curve such that $\delta(0)=w \in \partial A_1$, $\delta(0,1] \subset A_1$ and $\delta(1)=v$.  (Recall that most boundary points are accessible \cite[Sect.~6.4]{pom}).   Since $f:A_0\to A_1\setminus \{v\}$ is a  covering map, the set of preimages of $\delta$ in $A_0$ consists of  countably many disjoint simple curves $\{\delta_i\}_{i\in\Z}$, each connecting $\infty$ to a preimage $w_i\in\partial A_0 \cap \C$ of $w$.  No $\delta_i$ can have both ends at infinity, because otherwise $w$ would be an asymptotic value in $\partial A_1$ which contradicts the hypotheses. The curves $\delta_i$ divide $A_0$ into infinitely many fundamental domains $D_i$, each of which maps conformally onto $A_1\setminus \delta$. Hence the inverse branches $\psi_i:A_1\setminus \delta \to D_i$ are well defined and univalent.

Let $B\in A_1$ be a simply connected neighborhood of $v$ such that $\delta$ crosses  $\partial B$ exactly once.  Suppose that $f^{-1}(B)\cap A_0$ has two (and for simplicity, only two) different components $T$ and $T'$; that is, there exist two asymptotic tracts of $v$ in $A_0$. Our goal is to show that this cannot happen under the assumption that there no singular values in $\partial A_1$. 

To that end, start by observing that $T$ and $T'$ are both unbounded simply connected sets whose boundary is an open curve that tends to  infinity in both directions. Together, $T$ and $T'$ form the full preimage of $B$ in $A_0$. Hence the preimages of $\delta\cap B$ are unbounded subsets of the curves $\delta_i$ which map to $B$, and  thus belong either to $T$ or to $T'$.  Now let $\{\delta^T_j\}_{j\in\Z}$ and $\{\delta^{T'}_k\}_{k\in\Z}$ denote the two biinfinite subsets of $\{\delta_i\}_{i\in\Z}$ that  tend  to infinity within $T$  or $T'$ respectively, and indexed in such a way that $\delta^T_j$ is in between $\delta^T_{j-1}$ and $\delta^T_{j+1}$ (respectively for $T'$).  Also relabel the domains $D_i$ and the points $w_i$ accordingly, so that $D_j^T$ is bounded by $\delta^T_j$ and $\delta^T_{j+1}$, and $w_j^T\in \partial A_0$ is the finite endpoint of $\delta_j^T$ (respectively for $T'$). See Figure \ref{doubletract}.

\begin{figure}[hbt!]
\captionsetup{width=0.85\textwidth}
\centering
\includegraphics[width=\textwidth]{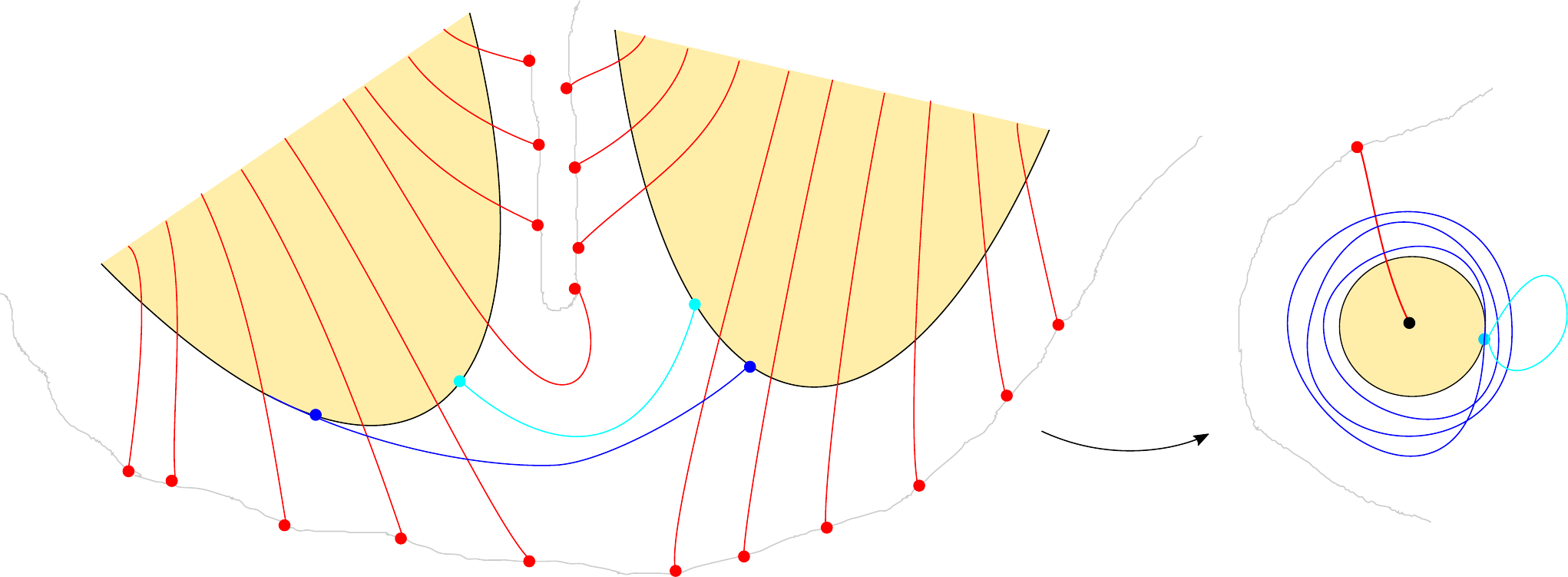}
\setlength{\unitlength}{\textwidth}
\put(-0.92,0.23){\large $T$}
\put(-0.38,0.32){\large $T'$}
\put(-0.68,-0.02){\scriptsize $w^T_J$}
\put(-0.625,0.18){\scriptsize $w^T_{J+1}$}
\put(-0.76,-0.0){\scriptsize $w^T_{J-1}$}
\put(-0.58,-0.02){\scriptsize $w^{T'}_K$}
\put(-0.54,-0.01){\scriptsize $w^{T'}_{K+1}$}
\put(-0.57,0.24){\scriptsize $D^{T'}_{K-1}$}
\put(-0.78,0.23){\scriptsize $D^T_{J}$}
\put(-0.84,0.29){\scriptsize \r $\delta^T_{J}$}
\put(-0.81,0.32){\scriptsize \r $\delta^T_{J+1}$}
\put(-0.55,0.34){\scriptsize \r $\delta^{T'}_{K-1}$}
\put(-0.5,0.33){\scriptsize \r $\delta^{T'}_{K}$}
\put(-0.1,0.12){$B$}
\put(-0.095,0.16){$v$}
\put(-0.12,0.245){\small \r $\delta$}
\put(-0.15,0.288){\small $w$}
\put(-0.97,0.13){\large $A_0$}
\put(-0.07,0.28){\large $A_1$}
\put(-0.07,0.15){\scriptsize $b$}
\put(-0.8,0.075){\scriptsize $b^T$}
\put(-0.54,0.1){\scriptsize $b^{T'}$}
\put(-0.65,0.055){\small $\gamma$}
\put(-0.63,0.1){\small $\gamma'$}
\put(-0.63,0.02){\large $C_2$}
\put(-0.67,0.36){\large $C_1$}
\put(-0.73,0.14){\scriptsize $b_J^T$}
\put(-0.56,0.19){\scriptsize $b_{K-1}^{T'}$}
\put(-0.03,0.2){\scriptsize $f(\gamma')$}
\put(-0.17,0.08){\scriptsize $f(\gamma)$}
\put(-0.29,0.06){$f$}

\caption{\small \label{doubletract} Sketch of the proof of Proposition \ref{mapprops}. In light gray we trace part of the boundaries of $A_0$ and $A_1$ although they play no role in this proof.} 
\end{figure}

Now consider a point $b\in \partial B$, $b\notin \delta$, and let $b^T$ and $b^{T'}$ be two arbitrary preimages (out of the infinitely many) in $\partial T$ and $\partial T'$ respectively. Let $\gamma$ be a simple curve joining $b^T$ and $b^{T'}$ through the open set $A_0\setminus (\overline{T\cup T'})$.  Then $\gamma$ divides this set into two unbounded components, say $C_1$ and $C_2$, whose union must contain the whole boundary of $A_0$. In particular, the points $w_i^T$ and $w_i^{T'}$, each belong either to $C_1$ or to $C_2$. 

Our first claim is that not all these points can belong to the same component $C_1$ or $C_2$. Indeed if they did, infinitely many curves $\delta_j^T$ and $\delta_j^{T'}$ would intersect $\gamma$. This would in turn imply that $f(\gamma) \in A_1\setminus B$ crosses $\delta$ infinitely many times and hence runs around $v$ infinitely many times and has infinite length. But this cannot happen because $f(\gamma)$ is a finite length closed curve with initial and endpoint $b\in\partial B$.   We conclude that $\gamma$ crosses only finitely many curves $\delta_i$ and thus there are infinitely many points $w_j^T$ and $w_k^{T'}$ in each component $C_1$ or $C_2$,  as claimed.

Choose one of the components, say $C_1$, and let $J, K\in \Z$ be the ``last" indices for which $w_J^T$ and $w_K^{T'}$ belong to $C_2$ or, in other words, $w_J^T, w_K^{T'} \in C_2$ but $w_j^T, w_k^{T'} \in C_1$ for all $j>J$ and $k< K$. This in turn implies that the curves $\delta_j^T$ and $\delta_k^{T'}$ do not intersect $\gamma$ for any $j>J$ or $k< K$; note that if they do, because both their endpoints lie on $C_1$,  $\gamma$ can be modified slightly to avoid such  intersections.

Finally observe that the domains $D_J^T$  and $D_{K-1}^T$ must then be part of the same fundamental domain, which is ``double" in some sense. Indeed, one could join the preimage $b_J^T$ of $b$ in $\partial T\cap D_J^T$ to the preimage $b_{K-1}^{T'}$ of $b$ in $\partial T'\cap D_{K-1}^{T'}$ by a curve $\gamma'$ not crossing any curve $\delta_j^T$ or $\delta_k^{T'}$. This would imply that $f(\gamma') \subset A_1 \setminus \delta$ and therefore $\gamma' = \psi_i(f(\gamma'))$ for some fixed $i\in\Z$. But $\gamma'$ contains at least two preimages of the point $b$ which contradicts the fact that  the inverse branch $\psi$ is univalent, and ends the proof of  part (c). 
\end{proof}

This concludes the proof of Theorem A.

\subsection{Boundaries of  basins of attraction}

In the next proposition,  we make the additional restrictions that    $\partial \bA$ is locally connected and contains no singular values to obtain properties of the component of the basin that maps to a neighborhood of the asymptotic value.   Under the additional assumption that $f \in \pmf$, we get a complete understanding of the topology and dynamics of the whole basin for these maps. 

\begin{proposition}[Bounded components] \label{lc}
Let $f$ be an entire or  meromorphic function  with an attracting cycle of period $p\geq 1$. Let $\bA$ be its immediate attracting basin
and suppose it contains a unique asymptotic value $v$  and no critical points. Assume $\partial \bA$ is locally connected and $S(f)\cap \partial \bA=\emptyset$. 
Let $A_0, A_1, \ldots, A_{p-1}$ be the components of $\bA$, named so that  $f(A_i)=A_{i+1 \pmod p}$ and $v\in A_1$. Then 
$\partial A_0$ is connected and unbounded. 

If, moreover, $ f \in \pmf$,   there is a unique pole in $\partial \bA$ which is in $\partial A_{p-1}$, and  for all $j=1, \ldots, p-1$, 
 $A_j$ is bounded.
 \end{proposition}

\begin{proof}
Recall from the propositions above that all components of $\bA$ are simply connected, $A_0$ is unbounded and maps infinite  to one to $A_1$, and $f: A_i \to A_{i+1\pmod p}$ is one-to-one for $i\neq 0$.  
The assumption of local connectivity of $\bA$, implies that $A_i$ is bounded if and only if  $\partial A_i$ is a single closed curve. Otherwise  $\partial A_i$ is a countable union of continuous unbounded curves. 

We first show that $\partial A_0$ is a single continuous curve with both ends at infinity. Indeed,  if $\partial A_0$ had at least two components, an argument analogous to the proof of Proposition \ref{mapprops} would show that at least one of the fundamental domains $D_j$ would be bi-infinite and thus map $2$ to $1$ onto $A_1$, contradicting the injectivity of $f$ on $D_j$. 

Thus $\partial A_0$ is connected and since $A_0$ is simply connected,  $A_0$ has only one access to infinity; that is, if $B(0,R)$ is a disk of radius $R$, then for all large $R$, $A_0 \setminus B(0,R)$ consists of a single unbounded component $C$.

Now assume $f \in \pmf$ and let $\gamma(t)\to\infty$ in $A_0$. Then $f^{-1}(\gamma(t))$ must converge to a pole, say $P\in\partial A_{p-1}$ or otherwise $\infty$ would be an asymptotic value. Let $N$ be the neighborhood of $P$ which maps onto $\C\setminus B(0,R)$ for some large $R$. Then $N\cap A_{p-1}$ contains the full preimage of  $C$. Now suppose there is a second pole $P'\in\partial A_{p-1}$ and choose a curve $\sigma(t)\to P'$ in $A_{p-1}$.  Then $f(\sigma(t))$ must tend to infinity in $C$ since there is no other access to infinity.  Let $N'$ be a preimage of $\C\setminus B(0,R)$ containing $P'$.   Choosing $R$ larger if necessary, $N$ and $N'$ are disjoint.  But  $N'\cap A_{p-1}$ must contain preimages of all points in $C$, which implies $f:A_{p-1} \rightarrow A_0$ is not one to one. Thus there is no second pole on $\partial A_{p-1}$.

Now suppose that $A_{p-1}$ is unbounded and let $\gamma(t)\to\infty$ within  $A_{p-1}$. Since we can choose $N$ so that $\gamma \notin N$, $f(\gamma)$ cannot enter $C$ and so must converge to a finite boundary point which is therefore an asymptotic value, contradicting the hypothesis that $f \in \pmf$.  This proves that $A_{p-1}$ is bounded. 
By the same argument, since $\partial \bA$ contains no asymptotic values, it follows that each component $A_i$, $i=1,\ldots,p-2$ is bounded by a continuous closed curve without poles.  
\end{proof}

\begin{remark}
In this proposition we   assumed  local connectivity of the boundary of the immediate basin  and concluded that all its components  except  $A_0$ are bounded. It is plausible that a partial converse statement is also true.  If we assume that $A_1$ is bounded, and add some hyperbolicity condition,  it is very possible that $\partial A_1$ locally connected. Then the same would be  true for  the remaining components.
\end{remark}

\section{Parameter space: Dynamically natural slices} \label{slices}

Families of functions defined by explicit formulas such as rational functions,  tangent or exponential functions clearly have ``natural'' embeddings into $\C^n$ for an appropriate $n$ in terms of their coefficients or  singular  values.    Up to affine conjugation, the images of these embeddings can be thought of as parameter spaces and we can study how the dynamics depends on these parameters.  We do not necessarily have closed forms for the families we are discussing here.   Nevertheless, there is a sense in which  every transcendental function $f$ with $N<\infty$ singular values belongs to an $N$ dimensional complex analytic manifold \cite{gold-keen2,EL92,bkl4}.

To make this precise we use the  theory of holomorphic motions for holomorphic families of entire and meromorphic maps.  For rational maps, the theory is explained in \cite{mcmbook} and it is adapted to meromophic functions in \cite{KK}.  We state the results from the latter reference that we use here.

\begin{definition}[Holomorphic family]
A {\em holomorphic family} of entire or meromorphic maps over a complex manifold $X$  is a  map
$\calf:X \times \C \rightarrow \hat\C$, such that $\calf(x,z)=:f_x(z)$ is meromorphic for all $x\in X$ and $x \mapsto f_x(z)$ is holomorphic for all $z\in \C$.   If $X$ has dimension $n$,  we say the family has dimension $n$ and $X$ is the parameter space for $\calf$. 
\end{definition}

\begin{definition}[Holomorphic motion]
A {\em holomorphic motion} of a set $V \subset \hat\C$ over a connected complex manifold with basepoint $(X,x_0) $ is a map  $\phi: X  \times V \rightarrow \hat\C$ given by 
 $(x,v) \mapsto \phi_x(v) $ such that
 \begin{enumerate}
\item  for each $v \in V$ , $\phi_x(v)$ is holomorphic in $x$, 
 \item
  for each  $x \in X$,  $\phi_x(v) $ is an injective function of  $v \in V$, and,
  \item  at $x_0$, $\phi_ {x_0} \equiv {\rm Id}$.
  \end{enumerate}

A holomorphic motion  of a set $V$ {\em respects the dynamics} of the holomorphic  family $\calf$ if $\phi_x(f_{x_0} (v)) = f_x(\phi_x(v))$  whenever both $v$ and $f_{x_0}(v)$ belong to $V$.
\end{definition}

Because maps in $\calf$ that are conjugate by an affine map have the same dynamics, we want to restrict ourselves to one representative for each conjugacy class.  We  do this by choosing an appropriate normalization for the functions and restricting $\calf$ to those normalized functions.  Below, we always assume the family is normalized in some way;  the normalization   depends on the family.   

The following equivalencies are proved for rational maps in \cite{mcmbook} and extended to the transcendental setting in \cite{KK}.

\begin{theorem} \label{jstable}
Let $\calf$ be a  holomorphic family of normalized entire or meromorphic maps with finitely many singular values, over a complex manifold $X$, with base point $x_0$. Then the following are equivalent.
\begin{enumerate}[{\rm (a)}]
\item The number of attracting cycles of $f_x$ is locally constant in a neighborhood of  $x_0$.
\item There is a holomorphic motion of the Julia set  of $f_{x_0}$ over a neighborhood of $x_0$ which respects the dynamics of $\calf$.
\item If in addition, for $i=1, \ldots,N$, $s_i(x)$  are holomorphic maps parameterizing the singular values of $f_x$, then the  functions $x \mapsto f_x^n(s_i(x))$ form a normal family on a neighborhood of $x_0$. 
\end{enumerate}
\end{theorem}

\begin{definition}[$J-$stability]
A parameter $x_0\in X$ is a $J$-stable parameter for the normalized family $\calf$ if it satisfies any of the above conditions.
We denote by $X^{stab}$, the set  of $J$-stable parameters for the family $\calf$.  Its complement 
\[
\calb_X:=X\setminus X^{stab}
\]
is known as  the {\em bifurcation locus} of the family $\calf$, and the elements in $\calb_X$ are the  {\em bifurcation parameters}.
\end{definition}

In families of maps with more than one singular value,  it makes sense to consider subsets of the bifurcation locus where only some of the singular values are bifurcating, in the sense that the  families  $\{g_n^i(x):= f_x^n(s_i(x))\}_n$ are normal in a neighborhood of $x_0$ for some values of $i$, but not for all.  We define
\[
\calb_X^i:= \calb_X(s_i)= \{x_0\in X \mid  \{g_n^i(x)\} \text{\ is not normal in any neighborhood of $x_0$ in $X$}.\}
\]
This is also known as the {\em activity locus (in $X$)} of the singular value $s_i$ (see \cite{dujfav,gaut}). 

In this paper we investigate one dimensional slices of holomorphic families in which the activity loci of the different singular values are disjoint. This means that at any given parameter on our slice only one of the singular values is allowed to bifurcate at a given parameter value.  More precisely, if the maps in our slice  have $N$ distinct singular values, we require the existence of $N-1$ persistent attracting or parabolic cycles of fixed multiplier throughout the slice. Singular values can take turns in ``serving'' these basins but only one of them at a time is allowed to be free.  These and other conditions are collected in the definition below of a {\em dynamically natural slice}.

\begin{definition}[\bf Dynamically natural slice]\label{def:dns}
 Let $\calf$ be a  holomorphic family of normalized entire or meromorphic   maps  over $X$.  Assume $S(f_x)$, the set of finite singular values of $f_x$  has cardinality $N<\infty$ for all $x\in X$.
 A one dimensional subset $\Lambda \subset X $ is a {\em dynamically natural slice} with respect to $\calf$ if the following conditions are satisfied. 
\begin{enumerate}[{\rm(a)}]
\item 
 $\Lambda$ is embedded or conformally equivalent to $\C$  minus a discrete set.  The removed points are called {\em parameter singularities}.  (By abuse of notation we denote its image in $\C$ by $\Lambda$ again,  and denote the variable in $\C$ by  $\lambda$, and the function $f_x$ by $f_{\la}$.

\item The singular values are given by distinct holomorphic functions $s_i(\lambda)$, $i=1,\ldots, N-1$, and an asymptotic value $v_\la:=v(\la) \not\equiv s_i(\la)$ for any $1\leq i \leq N-1$, that  is an affine function of $\la$. We further assume that no preimage of $v_\la$ is a critical point for all $\lambda$. We call $v_\la$ the {\em free asymptotic value}, and  we require that $\calb_\Lambda(v_\la) \neq \emptyset$ and $\calb_\Lambda(v_\la) \neq  \Lambda$. 
\item The poles (if any) are given by distinct holomorphic functions  $\{P_i(\lambda)\}_{0\leq i \leq M\leq \infty}$, $\lambda \in \Lambda$.
\item  
 There are $N-1$ distinct attracting  or parabolic cycles whose period and multiplier are constant for all $\la\in \Lambda$. 

\item  \label{propconj} Suppose there exists a parameter $\la_0$ such that  $v_{\la_0}$ is the only singular value in $\cala_{\la_0}$,  the basin of attraction of an attracting  cycle whose multiplier is not a constant function of $\lambda$. Then 
the slice $\Lambda$ contains, up to affine conjugacy,  all meromorphic maps $g:\C\to \chat$ that are quasiconformally conjugate to  $f_\lo$ in $\C$ and conformally conjugate to $f_\lo$ on  $\C \setminus \cala_{\la_0}$. 
\item $\Lambda$ is maximal in the sense that if $\Lambda'= \Lambda \cup \{\la_0\}$ where $\la_0$ is a parameter singularity, then $\Lambda'$ does not satisfy at least one of the conditions above. 
\end{enumerate}
\end{definition}

From now on, as long as it is understood from the context,  we will drop the index $\Lambda$. In other words $\calb :=\calb_\Lambda$.

\begin{definition}[\bf Dynamically natural family]\label{def:dnf}
We define  the subfamily of the holomorphic family $\calf$ of normalized entire or meromorphic functions,   $\calf_{\Lambda} = \{f_{\la} \}_{ \la \in \Lambda}$, as those  $f_{\la} \in \calf$ for which $\la$ lies in the dynamically natural slice $\Lambda$.  We call  $\calf_\Lambda$ (or $\{f_\la\}_{\la\in\Lambda}$)   a {\em dynamically natural family} for the slice $\Lambda$.   When referring to the parameters we use $\Lambda$ and when referring to the functions we use $\calf_{\Lambda}$.  
\end{definition}

Let us make some remarks on the conditions above. 

\begin{remark} \label{relax}

\begin{enumerate} [(i)] 
\item A given family $\calf$ may contain both entire and meromorphic functions. Condition (c) implies that in a dynamically natural slice the functions are either all entire or all meromorphic.  
\item \label{provando}
Condition (d) could actually be weakened for our purposes. As an example, we could require only $N-2$ attracting or parabolic cycles of constant multiplier plus an orbit relation between two distinct singular values (e.g. symmetry, or the two orbits coinciding after some iterates.) The tangent family is an example. Although most of our results hold in these more general situations, for simplicity of exposition we choose to require condition (d) as stated.  We will, however, also consider more general dynamically natural slices in the examples in Section \ref{sec:examples2}. 
\item  Condition (e) is generally easy to verify  in concrete families. 
\item 
Condition (f) is imposed to avoid artificial parameter singularities. In general, these singularities occur because the functions become constant or  some of the singular values coalesce and/or some poles  coalesce and become critical points; either of these events will  create  parameter singularities. 
\item 
Suppose $\Lambda$ is a dynamically natural slice. For a given parameter value $\la_0\in \Lambda$,  if the free asymptotic value $v_{\la_0}$ is attracted to an attracting cycle, it follows that  $\la_0\in \Lambda^{stab}$ (indeed, if any of the cycles is parabolic, it must be persistent and therefore the map is $J-$stable, although not necessarily hyperbolic). Because there are $N-1$ attracting or parabolic basins which need to attract $N-1$ different singular values,  only one of the singular values can be active at any given parameter.  That is,  if we set $\calb^0:=\calb(v_\la)$, then for $i,j=0,1,\ldots,N-1$, 
\[
\calb^i\cap \calb^j =\emptyset 
\]
whenever $i\neq j$. If we relax the definition of a dynamically natural slice as indicated in  Remark~\ref{provando} above, then we must also allow the possibility that $\calb^i =\calb^j$ for some $i\neq j$. This is the case, for example, in the tangent family.

\end{enumerate}
\end{remark}
   
 Most of the  slices which have been systematically studied in the literature, like the exponential family $e^z + \lambda$ or  the tangent family $\la \tan z$,  are dynamically natural slices of the larger family of functions with two asymptotic values and no critical values.   In the Appendix we show that many other slices can be constructed by pre- or post- composing functions  with finitely many asymptotic values with rational maps,  and some examples are shown in Section \ref{sec:examples2}.

\section{Distinguished parameter values: Proof of Proposition B.}\label{Dist values}

Let  $\Lambda$ be a dynamically natural slice for a family of entire or meromorphic functions and let  $ \calf_\Lambda=\{f_\lambda \}_{\la \in \Lambda}$ be the corresponding dynamically natural family.   Recall that the definition implies that $\#{S(f_\la)}$ is constant and finite for all $\la\in\Lambda$.  Let $v_\la$ denote the free asymptotic value of $f_\la$.  Our goal is to discuss special parameter values for which the forward orbit of the free singular value  is finite.   We  restrict our discussion to parameters varying in  $\Lambda$ and to the bifurcation locus $\calb(v_{\la})$, although the definitions could be adapted for the full holomorphic family defined over $X$.

\begin{definition}[Misiurewicz parameters]
A parameter $\la\in\Lambda$ (or the map $f_\la$) is called a {\em $\Lambda$ slice Misiurewicz parameter} if   the free singular value $v_{\la}$   has the property that $f_\la^n(v_{\la})$ is a repelling or parabolic periodic point for some $n\geq 0$. 
\end{definition}

\begin{remark*}
In the definition above we assumed the periodic point $f_\la^n(v_{\la})$ to be  repelling or parabolic, because we want our Misiurewicz parameters belong to $\mathcal{B}(v_\la)$.  Indeed, otherwise $f_\la^n(v_{\la})$ would be attracting, Siegel or Cremer. In the first case the parameter $\la$ would be $J-$stable while the last two cases cannot occur since by definition of $\Lambda$, they would imply the presence of $N$ non-repelling cycles, and therefore $N$ singular values apart from $v_\la$. 
 \end{remark*}

For simplicity of exposition, throughout this paper we will refer to $\Lambda$  slice Misiurewicz parameters  as Misiurewicz parameters.

By definition, Misiurewicz parameters are solutions of 
 \begin{equation}\label{eq:mis}
f_\la^m(v_{\la})=f_\la^n(v_{\la})
\end{equation}
for some $m>n\in \N$.  If $f_\la$ is conjugate to a Misiurewicz map $f_\lap$ which satisfies Equation (\ref{eq:mis}) for certain values of $m,n\in \N$, then $f_\la$ satisfies the same equation for the same $m,n$.

The second kind of distinguished parameter value that we discuss is specific to meromorphic maps. A meromorphic function has  at least one pole that is not an omitted value.  The pre-images of the poles have a finite forward orbit and play an important role in the dynamics.

\begin{definition}[Order of a prepole]
A point $p\in\C$ is a {\em prepole of order $n > 0$} for $f_\la$ if $f_\la^k(p)$ is defined for $k<n$ and $f_\la^n(p)=\infty$. 
\end{definition}
\noindent With this definition a pole is a prepole of order 1.

For every $\la\in\Lambda$, the  poles form a discrete set so  they can only accumulate at infinity. Unless there are at most two poles and both are omitted values, which  is never the case for functions of finite type, Picard's theorem implies that  the prepoles of order 2 form an infinite set.  Their accumulation set is the set of poles and the point at infinity. Prepoles of order $n\geq 3$  
 accumulate both at prepoles of order $n-1$ and at infinity. It is well known that prepoles are dense in the Julia set of $f_\la$  \cite{bergweiler,bkl1}. 

We now look at the parameter plane and  consider parameters for which some iterate of the free asymptotic value is a pole. 

\begin{definition}[Virtual cycle] 
Let $f$  be a meromorphic map. A {\em virtual cycle} for $f$ of period $p\geq 2$ is a set of points $\{ a_1, a_2, \ldots, a_{p-1},\infty\}$, $a_i\in\C$ for $i=1,\ldots,p-1$,  where $a_1$ is an asymptotic value, $a_{p-1}$ is a pole and $f(a_{i})=a_{i+1}$ for $1\leq i \leq p-2$. In other words a virtual cycle is the forward orbit of an asymptotic value which is a prepole.  
\end{definition}

\begin{definition}[Virtual cycle parameter]
A parameter $\lambda\in \Lambda$ is called a {\em virtual cycle parameter of order $p\geq 2$}  if $f_\lambda$ has a virtual cycle of period $p$. 
\end{definition}

For $\la\in\ps$, let $\{P_i(\la)\}_{0\leq i \leq M\leq \infty}$ denote the set of poles of $f_\la$.    Recall that since we are working in a dynamically natural slice,  each $P_i(\la)$ is a distinct  holomorphic function of $\la$ throughout $\ps$. By definition, virtual cycle parameters of order $p\geq 2$ in $\Lambda$ are  $\la-$values that  are solutions of
\begin{equation} \label{eq:vc}
f_\la^{p-2}(v(\la)) = P_i(\la)
\end{equation}
for some $0\leq  i \leq M$,  where as usual $v(\lambda)$ denotes the  free asymptotic value of $f_\la$.

\begin{remark*}
Virtual cycle parameters are parameter values for which an orbit relation exists  in the sense of (\ref{eq:vc}). Hence if $\lambda$ is a virtual cycle parameter and $f_\la$ and $f_\lap$ are topologically conjugate, then $\la'$  must be  virtual cycle parameters of the same order. 
\end{remark*}

\begin{proposition}[Special parameters are unstable]  \label{biflocus}
Misiurewicz parameters  and  virtual cycle parameters  belong to the bifurcation locus $\calb(v_\la) \subset \calb_\Lambda$. 
\end{proposition}

\begin{proof}
Virtual cycle parameters are solutions of (\ref{eq:vc}) while Misiurewicz parameters are solutions of (\ref{eq:mis}).  Since the zeroes of holomorphic functions are discrete, for fixed $m,n$ and $p$ both equations have a discrete set of solutions. Hence, since both conditions are invariant under topological conjugacy, the maps must be structurally unstable.   In  addition,  since repelling and parabolic periodic points as well as poles  are in the Julia set,   the maps are not  $J-$stable. 
\end{proof}

\begin{proposition}[Density of special parameters] \label{dense}
Let  $\Lambda$ be a dynamically natural slice. Both the virtual cycle parameters   and the  Misiurewicz parameters   are dense in  $\calb(v_\la)$.  Moreover  $\calb(v_\la)$ has no isolated points in $\ps$.
\end{proposition}
\begin{proof}
 The proof is a standard application of Montel's Theorem   (See  e.g.  \cite[Thm 3.36, p.57]{beardon} and \cite[Thm.1.5, p.129]{cargam}.)  We begin with the virtual cycle parameters in slices of meromorphic maps.  Let $\lo\in  \calb(v_\la)$  and let $U$ be a neighborhood of $\lo$ in $\ps$. Let $P_i(\la)$, $i=1,2,3$ be three prepoles of order 3 which are distinct for every $\la \in U$.  These exist since the set of prepoles of order $3$ or larger   is infinite. Now observe that  because of the  assumption that $\lo\in\calb(v_\lambda)$,  $\{g_n\}$ is not  normal in $U$, and so there must be  some $\lambda \in U$ such that $g_n(\la)=f_\la^n(v_{\la})$  takes one of the values  $P_i(\la)$, for some $i=1,2$ or $3$ and some $n$.    
 
To show that Misiurewicz parameters are dense we use a repelling periodic point $a(\la)$  of period  3 (for example) for all $\la\in U$, and apply the above arguments to  the functions $\{a(\la), f_\la(a(\la)), f_\la^2(a(\la))\}$, and to  $g_n(\la):=f_\la^n(v_{\la})$. 

Note that these proofs actually show that the sets of virtual cycle parameters of order 4 or Misiurewicz parameters of preperiod 1 and  period 3 are dense in $\calb(v_\la)$. By choosing prepoles of a different order or periodic points of a different period, the same proof would show that fixing these numbers arbitrarily, the corresponding subsets are also dense in $\calb(v_\la)$. Since these subsets are disjoint, it follows that they must accumulate on each other and therefore $\calb(v_\la)$ has no isolated points.
\end{proof}

\section{Shell components. Proof of Theorems C and  D. }\label{Shell}

Let   $ \Lambda$ be a dynamically natural slice of a normalized family of entire or meromorphic maps and let $\calf_{\Lambda}=\{f_\lambda\}_{\la\in\Lambda}$ be the associated dynamically natural family. Let $v_\la$ denote the free asymptotic value of $f_\la$ which, we recall, depends affinely on $\la$.  

 We are interested in studying those stable components in the slice $\Lambda$ that reflect the behavior of  the free asymptotic value. 
\begin{definition}[Shell Component] \label{def:shell}
A set $\Omega \subset \Lambda $ is a {\em shell component}   if it is a connected component of the set
 \[
\calh:=\{ \la\in\ps \mid v_\la \text{\ is attracted to an attracting cycle with nonconstant multiplier} \}.
\]

\end{definition}

Observe that if $\la \in \calh$, then  all remaining singular values must belong to the $N-1$ basins of parabolic  or attracting  cycles with constant multiplier, so that $v_\la$ is the only singular value in the basin of attraction of the cycle with nonconstant multiplier.  As a consequence, parameters in $\calh$ are $J-$stable in $\Lambda$; that is 
\[
\calh \subset \ps\setminus \calb =  \Lambda^{stab},
\]
where as above,  $\calb=\calb_\Lambda$ is the bifurcation locus of the slice $\Lambda$. 

Since $\partial \Omega$ is in $\calb$, it follows that $\Omega$ is also a connected component of $\ps\setminus \calb$. 

 Shell components may be hyperbolic components in the usual sense --- or they may not be; for example, if some of the cycles with constant multiplier are parabolic then no map in $\Lambda$ is hyperbolic. Conversely, capture components, where one of the persistent cycles attracts $v_\la$ and another singular value, are hyperbolic components but are not shell components.

The following property follows directly  from the definition.

\begin{lemma}[The period is constant]
Let  $\calf_{\Lambda}$ be a dynamically natural family and let $\Omega \subset \Lambda$ be a shell component. Then the period of the attracting cycle to which $v_\la$ is attracted  is constant throughout $\Omega$. It is called the {\em period of $\Omega$.}
\end{lemma}
Since the dynamics of all but one of the attracting cycles of $f_\la$ are fixed as $\la$ varies in a shell component,   {\em the attracting cycle} for this shell component is the one whose {\em multipler}   varies with $\la$.   Below, when we talk about an attracting cycle or its multiplier, in relation to a shell component, it is these that we mean.

Every hyperbolic component of the parameter plane of the well known exponential family $E_\la=z\mapsto e^z + \la$ is an example of a shell component.  In this case, the multiplier of the attracting cycle of $E_\la$ is never zero. This is a general phenomenon for shell components.

\begin{lemma}[Nonzero multipliers]
\label{lem:nonzeromult}
Let  $\calf_{\Lambda}$ be a dynamically natural family and let $\Omega \subset \Lambda$ be a shell component.  Then the multiplier of the attracting cycle is nonzero for all $\la \in \Omega$. 
\end{lemma}

\begin{proof}
If the multiplier of the attracting cycle were zero for $\la^*\in \Omega$, the cycle would contain a critical point. Therefore, since $v_\la$ is not persistently critical, in a neighborhood of $\la^*$ there would be two singular values in the basin. This is impossible because we are in a shell component.
\end{proof}

The following theorem explains how shell components are  parameterized by the multiplier of the attracting cycle.
Let $\Omega$ be a shell component of period $k>0$. We define the multiplier map as 
\[
\rho=\rho_\Omega: \Omega\to \D^* 
\]
where $\rho(\lambda)$ is the multiplier of the attracting cycle which attracts $v_\la$. 

\begin{theorem}[Multiplier map is a covering]
\label{multcov}
Let  $\calf_{\Lambda}$ be a dynamically natural family and let $\Omega \subset \Lambda$ be a shell component.   Then the multiplier map of the attracting cycle, $\rho: \Omega \to \D^*$,  is a holomorphic covering. 
\end{theorem}

\begin{proof}
The proof uses {\em quasiconformal surgery} on the basin of attraction of the attracting cycle that attracts $v_{\la}$. For details see \cite{BF}.
Note that because $f_\la$ is a holomorphic family  both $\rho$  and its derivative are holomorphic functions of $\la$.  

Choose a parameter $\lo \in \Omega $ such that $\omega_0:=\rho(\lo) \in \D^*$. Let $k>0$ be the period of $\Omega$ and ${\bf a^0}=\{a_0^0, \cdots, a_{k-1}^0\}$ the attracting periodic orbit to which the free asymptotic value $v_\lo$ is attracted.  Let $\cala^0$ denote its  full basin of attraction and let $A_i^0$ be the connected component of $\cala^0$ containing $a_i^0$ for $i=0,\cdots,k-1$. Assume without loss of generality that $v_\lo\in A_1^0$.

Let $W \subset \D^*$ be a (simply connected) neighborhood of $\omega_0$. For any $\omega\in W$  we define a Beltrami form $\mu_\omega$ in $A_1^0$  invariant under $f_\lo^k$ as follows.   Since ${\bf a^0}$ is attracting, there is a biholomorphic map $\psi$ from a neighborhood $U$ of $a_1^0$ to the unit disk $\mathbb D$ such that $\psi(a_1^0)=0$ and $\psi(f_\lo^k(z))=\omega_0 \psi(z)$ for every $z \in U$.   In $\mathbb D$ choose an annular fundamental domain $R_0=\{|\omega_0| < |z| < 1\}$ for the map $\psi(z) \mapsto \omega_0 \psi(z)$.  Now define the surgery using standard arguments: Construct a quasiconformal map $\tilde\phi_\omega: \mathbb D \to \mathbb D$ such that it fixes the origin and the unit circle and it maps $R_0$ to $R_\omega = \{|\omega| < |w| < 1\}$, conjugating multiplication by $\omega_0$ to multiplication by $\omega$.  In addition, it should satisfy $\tilde{\phi}_{\omega_0}=Id$ and  $\tilde\phi_\omega$ should depend holomorphically on $\omega$.   To do this, first define the map on $R_0$ and then extend it by the dynamics to all of $\mathbb D$.  The Beltrami coefficient $\tilde{\mu}_\omega$ of $\tilde{\phi}_\omega$ lifts to a Beltrami coefficient $\mu_\omega$ on $U$ invariant under the map $\psi$ that satisfies  $\mu_\omega(f^k_\lo(z)) \frac{\overline{ (f^k_\lo)'(z)}}{(f^k_\lo)'(z)}=\mu_{\omega}(z)$.   Extend $\mu_\omega$ to  the full backward orbit of $U$ so that for any inverse branch $g_n$ of $f^{-n}_\lo$ defined on a component of $\cala^0$, $\mu_\omega({g_n}(z)) \frac{\overline{ {g_n}'(z)}}{{g_n}'(z)}=\mu_{\omega}(z)$.  Extend it to be identically zero on the complement of ${\cala^0}$ in $\C$.   A Beltrami coefficient with this property is said to be {\em compatible with $f_\lo$}.  
Notice that the Beltrami coefficient $\mu_\omega$ depends holomorphically on $\omega$.  

By the Measurable Riemann Mapping Theorem \cite{AB},  there is a   quasiconformal homeomorphism $\phi_\omega:\C \to \C$ with Beltrami coefficient $\mu_\omega$.   It is unique up to post-composition by an affine map and  depends holomorphically on $\omega$.  Moreover, $\phi_{\omega_0}$ is conformal. 

 At this point we recall that we fixed the affine conjugacy class    by normalizing all the functions in the family $\calf_\Lambda$  in the same way.  Different normalizations work and/or are more convenient for different slices --  for example, in some, but not all slices,  one could  ask that $f_\la$  fix two given points in $\C$, or ask that it fix one point and also fix its multiplier. 

We can chose a uniform normalization for $\phi_\omega$  that preserves the normalization for $\calf_\Lambda$  and define the new entire or meromorphic map
\[
F_\omega:= \phi_\omega \circ  f_\lo \circ \phi\inv_\omega.
\]
This map respects the dynamics:  it has an attracting periodic cycle ${\bf a}(\omega)=\phi_\omega(\bf a^0)$ of multiplier $\omega$. Moreover, $F_\omega$ is quasiconformally conjugate to $f_\lo$ in the respective basins of attraction and conformally conjugate to $f_\lo$ everywhere else.  Since $\phi_{\omega_0}$ is conformal and fixes two points it follows that it is the identity map. Hence $F_{\omega_0} = f_{\la_0}$.

Since $f_\la$ is in a dynamically natural slice, property (d) ensures that $F_\omega=f_{\la(\omega)}$ for some $\la(\omega)\in \Lambda$.  Observe that the asymptotic value of $f_{\la(\omega)}$ is $v_\omega=\phi_\omega(v_\lo)$ and it depends holomorphically on $\omega$. Since $v_\la=v(\lambda)$ is affine,  $\lambda=\lambda(v)$ is also affine and hence $\la(\omega)=\la(v_\omega)$ is holomorphic. Since the multiplier changes with $\omega$, the map is nonconstant and thus open. Finally, $\la(\omega_0)=\la_0$ and therefore $\la(W)\subset \Omega$.

This construction defines a holomorphic local inverse of the multiplier map $\rho$ in a neighborhood of any point $\omega_0=\rho(\lo) \in \D^*$. It follows that $\rho$ is a covering map from $\Omega$ to $\D^*$. 
\end{proof}

\begin{corollary}[Connectivity of shell components] \label{dichotomy}
Let $\Omega$ be a shell component in a dynamically natural slice $\Lambda$. Then one of the two following possibilities occurs:
\begin{enumerate}
\item[(a)] $\Omega$ is simply connected and $\rho:\Omega \to \D^*$ is a universal covering (hence of infinite degree), or
\item[(b)] $\Omega$ is conformally equivalent  to $\D^*$ and $\lim_{|\omega| \to 0} \rho\inv(\omega) = \lambda^*$ where $\lambda^* \notin \Lambda$ is a parameter singularity. 
\end{enumerate}
\end{corollary}
\begin{proof}
Since $\rho:\Omega\to \D^*$ is a covering,  it follows from Lemma~\ref{coverings} that either $\rho$ is a universal covering and hence  has infinite degree and $\Omega$ is simply connected, or $\rho$ has  finite degree and $\Omega$ is conformally equivalent to $\D^*$.   In this case, the point $\lambda^*:= \lim_{|\omega| \to 0} \rho\inv(\omega)$ cannot belong to $\calb(v_\la)$ because,  by Proposition \ref{dense}, $\calb(v_\la)$ has no isolated points. It follows from Lemma \ref{lem:nonzeromult} that $\rho$ cannot extend continuously to $\la^*$ because $f_\la$ has no superattracting cycles. Therefore $\la^*$ is a parameter singularity.
\end{proof}

\begin{corollary} [Quasiconformal conjugacy in shell components]
Let $\calf_{\Lambda}$ be a dynamically natural   family  and let $\Omega \subset \Lambda$ be a shell component.  Then, any  two maps $f_{\la}$ and $f_{\la'}$ with $\la,\la'\in \Omega$ are (globally) quasiconformally conjugate. 
\end{corollary}
\begin{proof}
This follows from the proof of Theorem \ref{multcov}. Indeed, let $\gamma:[0,1]\to \Omega$ be a continuous path with endpoints  $\la$ and $\la'$. Then $\sigma(t):=\rho(\omega(t))$  is a continuous path in $\D^*$ joining $\omega=\rho(\la)$ and $w'=\rho(\la')$. The image of $\sigma$ is a compact set so it can be covered by a finite number of open disks  $\{D_i\}_{1\leq i\leq N}$ in $\D^*$ each centered at a point $\sigma(t_i)$.  Using the construction in Theorem \ref{multcov}, we can construct local inverses of $\rho$ mapping $D_i$ to an open neighborhood $D'_i$ of $\gamma(t_i)$. 
 By compactness, $\gamma[0,1] \subset \cup_{i=i}^N D_i'$. Since all pairs of parameters in each $D_i'$ correspond to quasiconformally conjugate maps, it follows that $f_\la$ is quasiconformally conjugate to $f_{\la'}$.
\end{proof}

\subsection{Virtual centers of shell components. }

In standard rational dynamics, hyperbolic components have a distinguished point called {\em the center} of the component at which the attracting cycle is superattracting. This is not the case for the shell components we are studying, where the multiplier of a finite cycle cannot vanish. Nevertheless we shall see that there is a  unique point in the boundary of the shell component that plays the role of the center.

\begin{definition}[Virtual center] \label{def:vc}
Let $\calf_{\Lambda}$ be a dynamically natural   family  and let $\Omega \subset \Lambda$ be a shell component.   Let  $\rho:\Omega \to \D^*$  be the multiplier map of $\Omega$.  A point $\lambda \in \left( \partial\Omega \cap \Lambda\right) \cup \{\infty\}$ is a {\em virtual center} of $\Omega$ if there exists  a sequence $\lambda_n \to \la$, with $\lambda_n\in \Omega$, such that the multipliers $\rho(\la_n)$ tend to 0.
\end{definition}

Our goal in this section is  to describe the relation between virtual cycle parameters and virtual centers. To do this we need a preliminary result which will also be useful later on.

\begin{proposition}[Bounded or unbounded cycle] \label{indifferent} \label{tendpole}
Let $\Omega \subset \Lambda$ be a shell component of period $p\geq 1$ for a dynamically natural family $\calf_{\Lambda}$.    Let  $\la_n \in \Omega$ be such that $\la_n\to \las\in \partial \Omega  \cap \Lambda$.  Let $\ba^n=\{a_0^n,\ldots, a_{p-1}^n\}$ be the attracting cycle of period $p$ for $f_n:=f_{\la_n}$, attracting the asymptotic value $v(\la_n)$. 
\begin{enumerate}[\rm (a)]
\item Suppose the cycle $\ba^n$ stays bounded as $n\to \infty$.  Then, $f_{\las}$ has an indifferent cycle $\ba$  whose period divides  $p$ and $\ba^n\to \ba$ as $n\to \infty$.
\item Suppose that $|a_j^n| \to \infty$ as $n\to \infty$ for some $0\leq j <  p$. Then, $p>1$ and   $f^k(v(\la^*))$ is a pole for some $0\leq k \leq p-2$, thus $\las$ is a virtual cycle parameter.
\end{enumerate}
Consequently, if the maps are entire or if $p=1$, only case {\rm (a)} can occur.
\end{proposition}

For this and other results in this section we shall use the following lemma. Recall that $S(f)$ denotes the set of singular values of $f$  and $B(a,r)$ denotes the disk of center $a$ and radius $r$.

\begin{lemma}[{\cite[Lemma2.2]{ripsta99}}] \label{magic}
Let $f$ be a transcendental meromorphic or entire  function. Let $R_0>0$ be such that a periodic orbit of $f$ belongs to $B(0,R_0)$.  Let $p\geq 1$, and $R>R_0$ such that $\{f^k(S(f))\}_{0\leq k <p}\subseteq B(0,R)$, and $|z|, |f^p(z)| > R^2$. Then,
\[
\left| \left(f^p \right)'(z)\right| > \frac{|f^p(z)| \log |f^p(z)|}{16 \pi |z|}.
\]

\end{lemma}

\begin{proof}[Proof of Proposition \ref{tendpole}]
\noindent {\rm (a)}  
Since the cycle stays bounded and $\las\in\Lambda$, it follows that  for all $0\leq i\leq p-1$,  there exists a subsequence of  the sequence $a_i^n$   that  has a limit $a_i$ and that  this limit point is a fixed point of $f_{\las}^p$. This cycle must be non-repelling. Since $\las\in\partial \Omega$, it cannot be attracting.  It also  cannot    have multiplier 0 because if it did,     $\ba^*$ would be superattracting and some point in the limit cycle would be a critical point. But then, for  every $\la$ in a neighborhood of $\la^*$, the analytic continuation of this cycle would be attracting and its basin would contain a critical point, and thus  $\lambda^*$ would be in the interior of $\Omega$ and not on its boundary.  Hence  $\ba^*$ is an indifferent cycle and part (a) is proved.

\vspace{0.2cm}

\noindent {\rm (b)}   We will give a proof by contradiction.  We will  show that if the conclusion  in (b) does not hold, then for $n$ large enough  the multiplier of the cycle must be larger than one in modulus.   Set $f_*=f_{\las}$, $v_n=v(\la_n)$ and $v_*=v(\la_*)$.  

Let $c_*$ be a periodic point of $f_*$ which can be analytically continued to $c(\lambda)$ in a neighborhood of $\las$, and choose $R_0$ so that $|f_n^k(c(\lambda_n))|, |f_*^k(c_*)|< R_0$ for all $k\leq p$ and  all $n>0$.  

Let us first deal with the case $p=1$ and show that the fixed point $a_0^n$ of $f_n$ cannot tend to infinity with $n$. Suppose it does. Let $R$ be large enough  to satisfy  $R>e^{8\pi}$ and also large enough that for $n>N$,  $B(0,R)$ contains $v_n$  and all the remaining singular values of $f_n$.  This is possible  because $\la* \in \Lambda$ so that  $f_*$ is well defined as a limit of $f_n$, and thus $v_n\to v_*$ and all the other singular values (finite in number) also depend holomorphically on $\lambda$. 

Now choose $N>0$ such that for $n>N$,  $|a_0^n|>R^2$ and recall that since $p=1$, $f_n(a_0^n)=a_0^n$. Then Lemma \ref{magic} implies that 
\[
|f_n'(a_0^n)|> \frac{\log |a_0^n|}{16 \pi} > \frac{\log R}{8\pi}>1,
\]
which contradicts the assumption that $\la_n\in\Omega_n$  for all $n>0$.

Next assume that $p>1$. Suppose that the conclusion in part (b) is not true so that $f_*^k(v_*)$ is not a pole for any $k<p-1$. Then, for all $0\leq k\leq p-1$, $f_*^k(v_*)$ is well defined and  $f_n^k(v_n)\to f_*^k(v_*)$ as $n\to \infty$.  Since all the singular values of $f_n$ and $f_*$, other than the free asymptotic value, vary continuously with $\la$ and are attracted to parabolic or attracting cycles of fixed multiplier,   there exists $R>\max\{R_0,e^{8\pi}\} $ large enough, and independent of $n$, such that for all $n$, $B(0,R)$ contains the first $p-1$ iterates of all these singular values; that is 
\[
\bigcup_{0\leq k < p} f^k(S(f_n)) \subset B(0,R), \quad \text{for all $n>0$}.
\]

Since  by hypothesis $|a_j^n| \to \infty$ as $n\to\infty$, and $f_n^p(a_j^n)=a_j^n$,  there exists $N>0$ such that  $|a_j^n|>R^2$ for all $n>N$. Hence the hypotheses of Lemma \ref{magic} are satisfied for every $n>N$ and we conclude that 
\[
\left| \left(f_n^p \right)'(a_j^n)\right| > \frac{ \log |a_j^n|}{16 \pi} > \frac{\log R}{8\pi}>1.
\]
It follows that the multiplier of the cycle $\ba_n$ is larger than 1 for $n>N$ which contradicts the hypothesis that  $\la_n\in\Omega$ for all $n>0$. 
\end{proof}

With these results in hand, we can prove that virtual centers are always virtual cycle parameters. 
\begin{theorem}[Virtual centers are virtual cycle parameters] \label{vc1}
Let $\calf_{\Lambda}$ be a dynamically natural   family  and let $\Omega \subset \Lambda$ be a shell component.    Let $\Omega \subset \Lambda$ be a  shell component  of period $p\geq 1$ and let $\la^*$ be  a parameter in $\partial \Omega \cap \Lambda$.  Then, if  $\las$ is a virtual center,  $\lambda^*$ is a virtual cycle parameter.

Consequently, if $\Lambda$ is a slice of entire maps, $\Omega\cap \Lambda$ has no virtual centers or, in other words, every virtual center is at infinity.
\end{theorem}

\begin{proof}
Suppose that  $\lambda^*\in\Omega\cap\Lambda$ is a virtual center. Then there exists a sequence of parameters $\la_n \in \Omega$  such that $\la_n\to \las\in \partial \Omega  \cap \Lambda$ and the multipliers $\rho(\la_n)$ tend to $0$. Let  $\ba_n=\{a_0^n, \cdots, a_{k-1}^n\}$ be the cycle  to which the free asymptotic value $v_n=v(\la_n)$ is attracted by the iterates of $f_{\la_n}$.  

We claim that, taking a subsequence if necessary,  one of the points in the cycle must tend to infinity as $n\to \infty$. Indeed, otherwise the cycle would converge to a finite cycle $\ba^*$ of $f_\las$ with  multiplier $0$, which  is impossible by Proposition \ref{indifferent} (a).

Hence one of the points in the cycle tends to infinity. It now follows from Proposition \ref{tendpole} (b) that $\lambda^*$ is a virtual cycle parameter. 
\end{proof}

\begin{remark}  \label{remconj}
It seems plausible to believe that a partial  converse is true. More precisely,  if a virtual cycle $\las \in \partial \Omega$ is accessible from $\Omega$, then $\las$ is a virtual center. The idea for a proof would be to find a new path in $\Omega$, $\tilde{\la}(t) \to \las$ along which the multiplier would tend to 0. This could be accomplished if a path  through a preasymptotic tract of $v(t)$  such that $v(t) \to v_*$ could be found. See Conjecture \ref{conj1}.
\end{remark}

Propositions \ref{tendpole} and \ref{vc1} above have an immediate consequence for components of period one.

\begin{corollary}[Period one shell components] \label{period1}
Let $\calf_{\Lambda}$ be a dynamically natural   family  and let $\Omega \subset \Lambda$ be a shell component.   Let $\Omega$ be a shell component of period one. Then, for any sequence $\la_n\to \la^* \in \partial \Omega \cap \Lambda$, the attracting fixed point $a_n$ converges to an indifferent fixed point $a^*$ of $f_{\la_*}$. Moreover, $\partial \Omega \cap \C$ has no virtual cycle parameters and consequently no virtual centers.  
\end{corollary}


\subsection{The boundary of the shell component}
Let $\Omega$ be a simply connected shell component and as usual,  let $\lp$ denote the left half plane. Recall from Corollary \ref{dichotomy} that the multiplier map $\rho:\Omega \to \D^*$ is a universal covering. Thus, there exists a  biholomorphic map $\varphi:\lp \to \Omega $, unique up to precomposition by a M\"obius transformation, such that 
\[
(\rho\circ \varphi)(w)= e^w.
\] 
  
\begin{theorem}[Extension to the boundary]\label{boundarycont}
Let $\calf_{\Lambda}$ be a dynamically natural   family  and let $\Omega \subset \Lambda$ be a shell component.     Let $\Omega$ be a simply connected shell component in $\Lambda$ and let $\rho$ and $ \varphi$ be as above.  Then,   $\varphi$ extends continuously to a map $\varphi:\overline{\lp} \to \overline{\Omega} $, where the closures are taken in $\chat$.  Moreover, the point $\varphi(\infty) \in \chat$ is a virtual center.
\end{theorem} 

 Given $a<b\leq\infty$, we say that an arc  $\gamma:(a,b)\to \C$ {\em lands} at $z_0\in\widehat{\C}$ if $\lim_{t\to b} \gamma(t)=z_0$ in the spherical metric. 

\begin{proof}[Proof of Theorem \ref{boundarycont}]
Let $p$ be the period of $\Omega$. We first deal with the case $p\geq 2$.

Let $\omega(t)$, $t\in (0,1)$, be an arbitrary arc in $\lp$ such that  $\lim_{t\to 1} \omega(t)=i \theta^*$ for some  $\theta^*\in \R$. We want to show that as $t \to 1$,  the accumulation set of $\la(t)=\varphi (\omega(t))$ in  $\partial \Omega =\overline{\Omega}\setminus \Omega$ is a single point.   That is, $\la(t)$ lands at a point in the boundary (which could be infinity) when $t\to 1$. 

To this end,  assume that $\la(t)$ does not land at infinity and let $\la^*$ be a finite accumulation point of $\la(t)$ as $t\to 1$. This means that there exists a sequence $t_n\to 1 $ as $n\to \infty$,  such that $\la_n=\la(t_n) \in\Omega$ satisfies $\la_n\to \las$. Let ${\bf a}^n$ denote the attracting cycle corresponding to $\la_n$ and set $f_*=f_{\la_*}$. 

If the attracting cycle ${\bf a}^n$ remains bounded for all $n>0$, it follows from Proposition \ref{indifferent} that $f_{*}$ has an indifferent cycle whose period divides $p$. 

On the other hand if the cycle does not remain bounded, let $n_k\to \infty$ be a subsequence for which some point in the attracting cycle tends to infinity. Since $\la_{n_k}\to \las$, it follows  by Proposition \ref{tendpole} that $\la^*$ is a virtual center.

Thus, all accumulation points of the curve  $\la(t)$ are either parameters with a neutral cycle of period dividing $p$ or virtual cycle parameters, and both these sets are discrete subsets of $\Lambda$.  Since,  the set of accumulation points must be connected \cite[p.33]{pom},  it follows that $\la(t)$ accumulates at a single point, i.e., that $\la(t)$ actually lands at a point $\las\in \partial \Omega$ (which might be infinity). 

Since $\omega(t)$ was arbitrary among all curves tending to $i\theta$, all of them have the same property. By Lindel\"of's theorem  \cite[Prop.~2.14]{pom}, all their images by $\varphi$ must land at the same point $\la^*$. Hence $\varphi$ has an  unrestricted limit at every point $i\theta$ in the imaginary axis, and $\varphi({i\theta}), \, {\theta\in\R},$ is a continuous curve running along $\partial \Omega$.  Note that this curve might pass through infinity several times. 

It remains to consider the case that $\omega(t)\to \infty$ in $\lp$.  Using  exactly the  same arguments as above, $\lambda(t)=\varphi(\omega(t))$ must land at a single point $\la_c \in \partial \Omega$ and hence the   limit of $\varphi$ at $\infty$ exists and is independent of the choice of $\omega(t)$. In particular, if we choose $\omega(t)$ to be the negative real axis, $\lambda(t)\to \la_c$ as $t\to 1$, and the multiplier $\gamma(\lambda(t))$ tends to zero. This shows that $\la_c$ is a virtual center and, by Theorem \ref{vc1}, also a virtual cycle parameter.

Since the unrestricted limit exists at every point of $\partial \lp$,  it follows that $\varphi$ extends continuously to $\overline{\lp}$. If any of the limits is the point at infinity, continuity has to be understood in the spherical metric. 

Finally suppose that $p=1$. The same arguments as above show that $\varphi$ has unrestricted limits at all points in $\partial \lp$. Nevertheless,  since in this case $\Omega$ has no finite virtual centers, $\lambda(1)=\infty$. 
\end{proof}

Using Caratheodory's Theorem \cite[Theorem 2.1]{pom} we obtain the following corollary. 

\begin{corollary}[Local connectivity]  \label{cor:uniquevc}
Let $\Lambda$ be a dynamically natural slice of a family of  entire or meromorphic maps.  Let $\Omega$ be a simply connected shell component in $\Lambda$.  Then $\partial \Omega$ is locally connected and has a unique virtual center. Moreover, if the period of $\Omega$ is one, this virtual center is at infinity and $\Omega$ is  unbounded.
\end{corollary}

\begin{remark}
If a shell component is not simply connected, the arguments above show that $\varphi$ extends continuously to all points in on the boundary but the the role of the virtual center may be taken by a parameter singularity. An example of this can be found in the tangent family $z\mapsto \lambda \tan(z)$, where $\D^*$  is a shell component of period one  and the origin is  a parameter singularity.  In this example, the multiplier map is biholomorphic.
\end{remark}

\begin{remark}
Our methods do not show that $\partial \Omega$ has a unique virtual cycle parameter  because we have not discarded the possibility that  while $\la(1)=\varphi(i\theta)=\las\in\C$ is a virtual cycle parameter, $\la^* \neq \la_c$,  the virtual center of $\Omega$. If such a parameter were to exist, it would need to be isolated. Nevertheless,  in view of Remark \ref{remconj} we believe that this should never happen. We state this as a conjecture.
\end{remark}

\begin{conjecture}  \label{conj1}
In the setup of Corollary \ref{cor:uniquevc}, $\partial \Omega$ has a unique virtual cycle parameter which is the virtual center. 
\end{conjecture}

 Theorem~\ref{boundarycont} also allows us to define an internal structure in the simply connected shell component $\Omega$.
\begin{corollary}[Internal rays] \label{cor:internalrays}
Under  the hypotheses of Theorem~\ref{boundarycont} and Corollary \ref{cor:uniquevc}, define the internal ray in $\Omega$ of angle $\theta$ by 
\[
R_\Omega(\theta):=\{ \varphi(t+2\pi i \theta),  t\in (-\infty,0)\}.
\]
Then, all internal rays have one end at the virtual center and the other end at a point in $\partial \Omega$ (which could be infinity). If the virtual center $\la^*$  is finite no internal ray has both ends at $\la^*$. 
\end{corollary}

Notice that the internal rays foliate the shell component.  The rays landing at parameters corresponding to parabolic cycles of multiplier one, that is  the preimages of $\{0,1\}$ under the multiplier map extended to $\partial\Omega$,  divide $\Omega$  into fundamental domains which are mapped one to one to the punctured disk by the multiplier map $\rho$. See Figure \ref{sketch}.

We have shown that simply connected shell components of period one are always unbounded.    In the next section we look at some examples of dynamically natural slices of parameter spaces for specific families.   The first is the  exponential family $\{E_{\lambda}=e^z+\lambda\}_{\la \in \C}$  where all hyperbolic components are unbounded.   Other examples are dynamically natural slices  $\Lambda$ of meromorphic families $\calf \subset \pmf$. In these parameter planes we observe that all of the simply connected shell components of period greater than one are bounded.  These observations lead us to the following conjecture. 

\begin{conjecture}[Boundness conjecture]
Let $\Lambda$ be a dynamically natural slice of the meromorphic family $\calf \subset \pmf$.  Let $\Omega \subset \Lambda$ be a simply connected shell component  of period $p  \geq 2$.   Then $\Omega$ is bounded. 
 \end{conjecture}
\begin{remark} If the slice is defined in the relaxed sense of Remark \ref{relax}, there may be a finite number of unbounded components of period greater than $1$ because of the relations among the asymptotic values.  An example is given by the tangent family, where a shell component of period two is unbounded. \end{remark}

\section{Examples}\label{sec:examples2}

Our results apply to very general families of entire and meromorphic functions.   To find examples, we need to look at families whose singular sets depend holomorphically on the parameters.  We also want to be sure that if we perturb a hyperbolic map in our family by a holomorphic motion, the resulting map is still in the family.  And finally, to compute a dynamically natural slice, we would like to have an explicit formula for the maps that contains the parameters as constants.  In the appendix, we show how to find fairly large classes of families with these properties. 

The examples in this section are all families of entire and meromorphic functions that have two asymptotic values and either no critical values or just one critical value.   That the formulas we use for the functions are valid under perturbation follows from the material in the appendix.

 \begin{example1}  Any  meromorphic function with exactly two asymptotic values and no critical points belongs to the family 
  $$\calf_2 = \left\{ \frac{a e^{z} + b e^{-z}}{ce^{z} + d e^{-z}}, \, \, a,b,c,d \in\C, ad-bc =1 \right\}. $$
 See the Appendix and the references there. The asymptotic values are $a/c$ and $b/d$.  If $b$ or $c$ is equal to zero, the function is the exponential function and infinity is an asymptotic value.  Otherwise both asymptotic values are finite.   Any function in the family is determined by three complex constants. 

 We will look at several dynamically natural slices of $\calf_2$  and use different normalizations for each. 


 \begin{enumerate}[\rm(a)]
  \item \label{ex 1a} The exponential family $E_\la(z)= e^z + \lambda$ is one of the best understood dynamically natural slices of $\calf_2$.  It is the slice defined by fixing one of the asymptotic values of the functions $f_{\la} \in \calf_2$ at infinity by setting $c=0$ and taking the other as the parameter $\la=b/d$.  We can put these functions into the form above, normalizing by an affine conjugation that replaces  $z$ by $2z$.  Choosing   the coefficients as $a=1/2, d=2, c=0$ and $b=\la$ we get  $E_{\la}(0)= 1 + \la$.    The infinite asymptotic value  has neither a forward nor backward images (poles).  Thus the dynamics of the functions in this slice are determined by the orbit of the other asymptotic value,  $v_\la=\la \in \Lambda=\C$.   
 Since every meromorphic map with these properties is affine conjugate to a member of this family, property \ref{propconj} holds.

 Note that another normalization for $\{E_\la\}_\la$ in common use is one that fixes the second asymptotic value at the origin.  The family then takes the form $a e^z$ and the parameter is an affine function of the first iterate of the finite asymptotic value.  Standard references include  \cite{bakrip,devfagjar,sch,rempethesis}. 

Figure~\ref{fig:exponential family} shows the parameter plane for $E_\la$.  It is known that all hyperbolic components  in $\Lambda$ are unbounded and simply connected and that the multiplier goes to zero as the parameter goes to infinity inside the component;  that is,  infinity acts as a virtual center for all the hyperbolic components.   
 In the figure, components are colored according to their respective period, yellow for 1, cyan-blue for 2, red for 3, brownish green for 4, etc.
\begin{figure}[htb!]
\captionsetup{width=0.85\textwidth}
\centerline{
\includegraphics[width=0.3\textwidth]{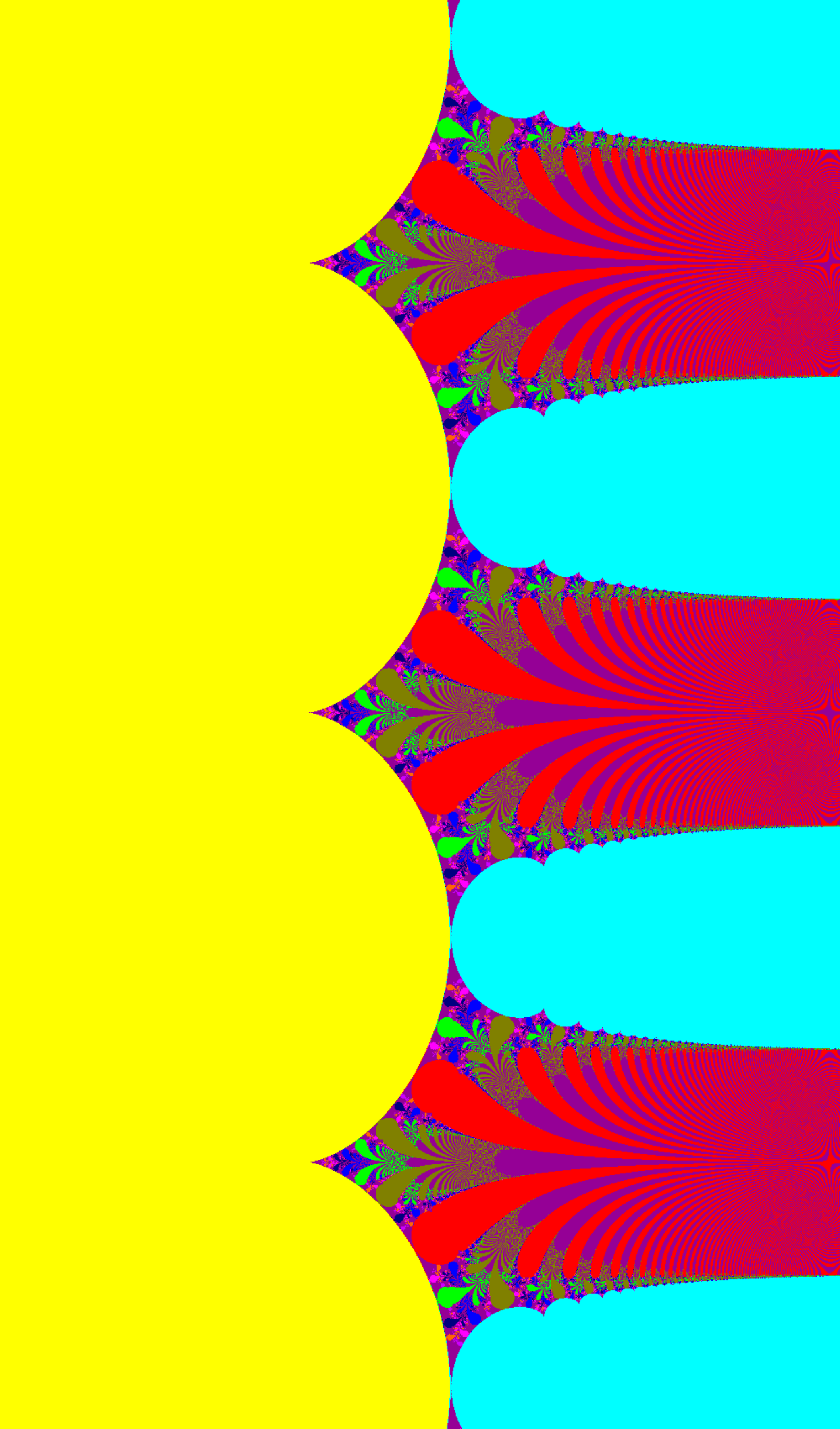}}
\caption{\label{fig:exponential family} \small The parameter plane of the exponential family $E_\la(z)=\exp(z)+\la$. Components are colored according to their respective period, yellow for 1, cyan-blue for 2, red for 3, brownish green for 4, etc.}
\end{figure}

  \item\label{ex 1b'}
   Another dynamical slice that can  be extracted from $\calf_2$ is formed by requiring the origin to be a fixed point with persistent multiplier $\rho_0, \, |\rho_0|<1$.   The resulting slice is $\Lambda=\C\setminus \{0,\rho_0/2\}$ and the maps have the form
 \[
f_{\lambda}(z)=\frac{e^z-e^{-z}}{(1/\lambda) e^z + (1/\mu) e^{-z}},
\]
with $\mu=\rho_0\lambda/(2\lambda-\rho_0)$. The asymptotic values are $\la$ and $-\mu$ and at least one of them is attracted by the origin.  
See \cite{CJK19,CJK20}  for a discussion of the topology of the capture component for maps in the slice where $\rho=2/3$ and  \cite{GaKo} for a discussion of properties of maps in the slice where $\rho_0=1/2$. To see that condition (e) is satisfied, observe that any conjugacy which is conformal in the basin of attraction of 0 keeps the multiplier of this fixed point unchanged.  Since the resulting conjugate map again has two asymptotic values and no critical points, it belongs to $\calf_2$. Using an affine conjugacy, we may require that the new fixed point  be the origin again;  thus the new map belongs to the slice.

Figure~\ref{fig:puremero2a} shows the parameter plane for for  $\rho_0=2/3$ and hence  $\mu=-\lambda/(1-3\lambda)$ and $\Lambda=\C \setminus \{0, 1/3\}$.  

\begin{figure}[hbt!]
\centering
\captionsetup{width=0.85\textwidth}
\includegraphics[width=0.4\textwidth]{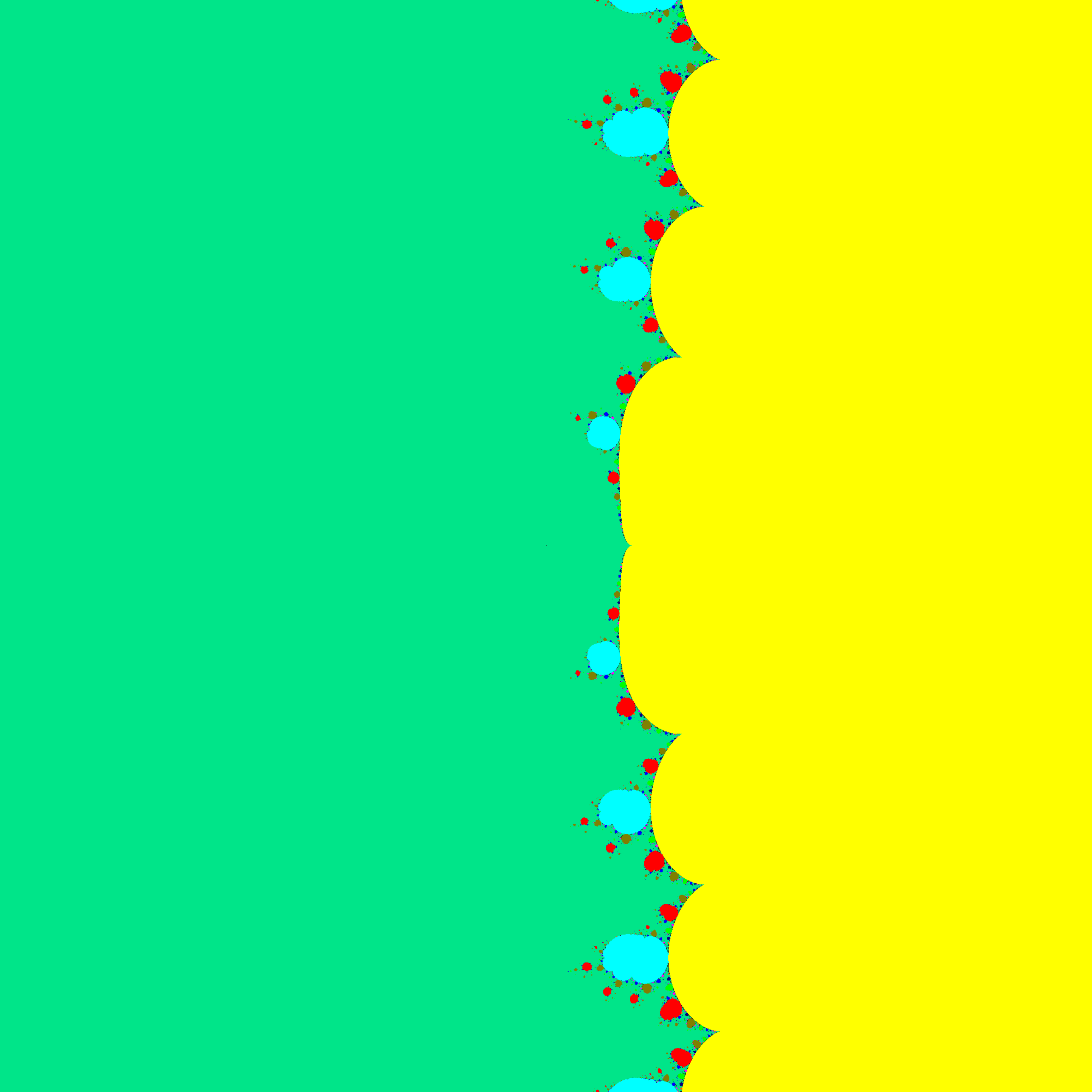} \hfil
\includegraphics[width=0.42\textwidth]{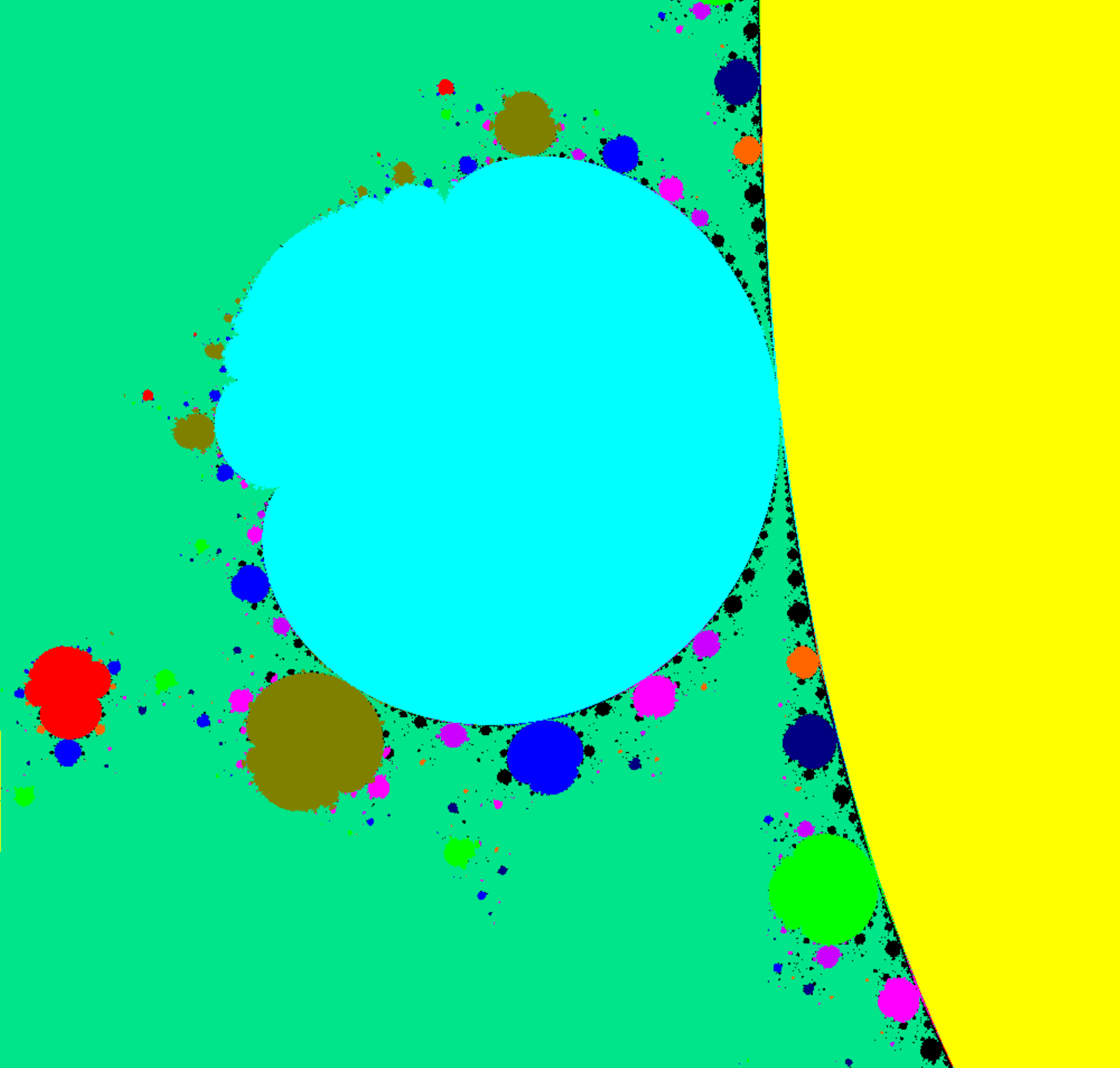}
\setlength{\unitlength}{0.9\textwidth}
	\put(-0.78,0.222){\circle*{0.005}}
	\put(-0.8,0.215){\scriptsize$0$}
	\put(-0.754,0.19){\line(1,0){0.028}}
	\put(-0.754,0.19){\line(0,-1){0.028}}
	\put(-0.754,0.162){\line(1,0){0.028}}
	\put(-0.726,0.162){\line(0,1){0.028}}
	\put(-0.726,0.176){\vector(1,0){0.25}}
\caption{\label{fig:puremero2a} \small Left: The $\la-$plane of the meromorphic  family $f_\la$, showing the dynamics of the free asymptotic value $\la$. Dimensions $[-12,12]\times[-12,12]$. Right: Zoom in on a shell component of period two (cyan-blue). Color coding is explained in the text. Dimensions $[0.5, 1.9]\times[-3.25,1.92]$}
\end{figure}

As in Figure~\ref{fig:exponential family}, the color represents the period of the attracting orbit that the free asymptotic value $\lambda$ is attracted to.   It does not reflect the behavior of the orbit of the other asymptotic value $-\mu$.   In the unbounded green capture component on the left  $\lambda$ is attracted to the origin.   As we see below, $-\mu$ may or may not also be attracted to the origin in this component.   In the yellow shell component  on the right, $\lambda$ is attracted to an attracting fixed point  different from zero while $-\mu$ is attracted to zero. It follows from Theorem C (c) that this component is indeed unbounded.   In the cyan-blue bounded shell components, $\lambda$ is attracted to a cycle of period $2$; in the red ones the attracting cycle has period three, etc.  Bifurcations occur at parabolic parameters as usual, giving rise to shell components of all periods attached to the boundary of any given shell component.

In Figure~\ref{fig:puremero2b} we   again plot the $\la-$plane but we show the dynamics of the second asymptotic value $-\mu$.    On the unbounded green  region of Figure~\ref{fig:puremero2b}, $-\mu$ is attracted to the origin.  This figure should be superimposed on top of Figure~\ref{fig:puremero2a} near the origin in the large green component on the left in Figure~\ref{fig:puremero2a}; it is relatively so small that it almost vanishes.   In intersection of the green regions of the the combined figure  both asymptotic values are attracted to the origin.   On the complement of the green region in  Figure~\ref{fig:puremero2a},  $\la$ is attracted to the origin and $-\mu$ is free. The yellow region is a shell component of period one which should have a virtual center at the parameter $\la=1/3$, (where $\mu=\infty$)  but this is a parameter singularity.   

Observe that $\mu$ is not an affine function of the parameter $\lambda$;  this would explain the numerics which seem to indicate that  this period one component is bounded. In fact, if we were to reparametrize the family so that $\mu$  depended affinely on the parameter, we would see the same picture as in  Figure~\ref{fig:puremero2b}  because of  the symmetry in the equation connecting $\mu$ and $\lambda$.

\begin{figure}[hbt!]
\centering
\captionsetup{width=0.85\textwidth}
\includegraphics[width=0.5\textwidth]{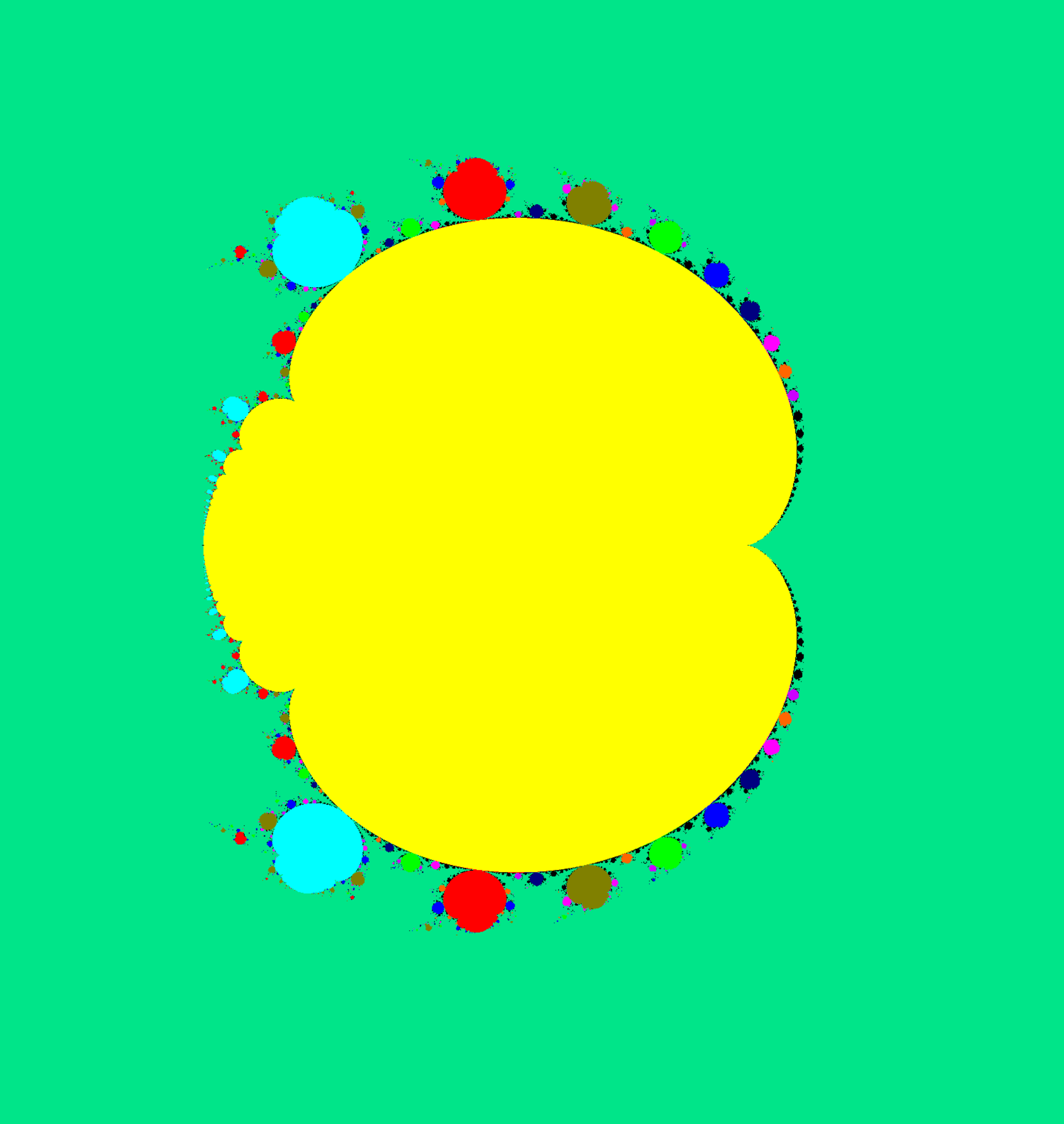}
\setlength{\unitlength}{0.5\textwidth}
	\put(-0.81,0.55){\circle*{0.02}}
	\put(-0.795,0.543){\scriptsize $\frac{1}{3}$}
\caption{\label{fig:puremero2b} \small The $\la-$ plane of the family $f_\la$, showing the dynamics of the second asymptotic value $-\mu=\la/(1-3\la)$. Dimensions $[0.28,0.44]\times[-0.08,0.08]$.}
\end{figure}

\item\label{ex 1c'} 
A slice  of $\calf_2$  is given by  the tangent family of maps 
$$T_\la(z)=\la \tan z$$  with $\Lambda=\C^*$ and $v_\la=\la i$ \cite{DK,KK,KY}.  This is a  dynamically natural slice in the  relaxed sense of Remark \ref{relax}  because the second asymptotic  value is      $s_1(\la)=- v_\la=-\la i$ and $T_\la$ is symmetric. It follows that the bifurcation locus is the same for both singular values.      Condition (e) should be replaced by the condition that the relation $s_1(\la)=- v_\la=-\la i$ must persist under deformation.   
Figure~\ref{fig:tan} shows the parameter plane of $T_\la$.\footnote{ Note that the maps $\lambda \tan{z}$  and $\lambda \tanh {z}$ are conjugate under $z \to iz$ so they are dynamically the same.} 
  There is a lot of symmetry in this family. We have
$$T_{\lambda}(-z) = T_{-\lambda}(z)  \mbox{   and }  T_{\lambda}(-z)= -T_{\lambda}(z) \mbox{   so that } $$
$$ T_{-\lambda}(-z) = T_{\lambda}(z).$$
In addition
$$ T_{\bar\lambda}(z)=\overline{T_{\lambda}(\bar{z})}. $$
This says, for example, that if for  some $\lambda_0$,  $z_1$ is a fixed point of $T_{\lambda_0}$,  then $T_{-\lambda_0}(z_1) = -z_1$  so that $z_1$ is a period $2$ point of $T_{-\lambda_0}$.  It also says that  complex conjugate values of $\lambda$ have conjugate periodic cycles.
   
\begin{figure}[htb!]
\centering
\captionsetup{width=0.9\textwidth}
\includegraphics[width=0.5\textwidth]{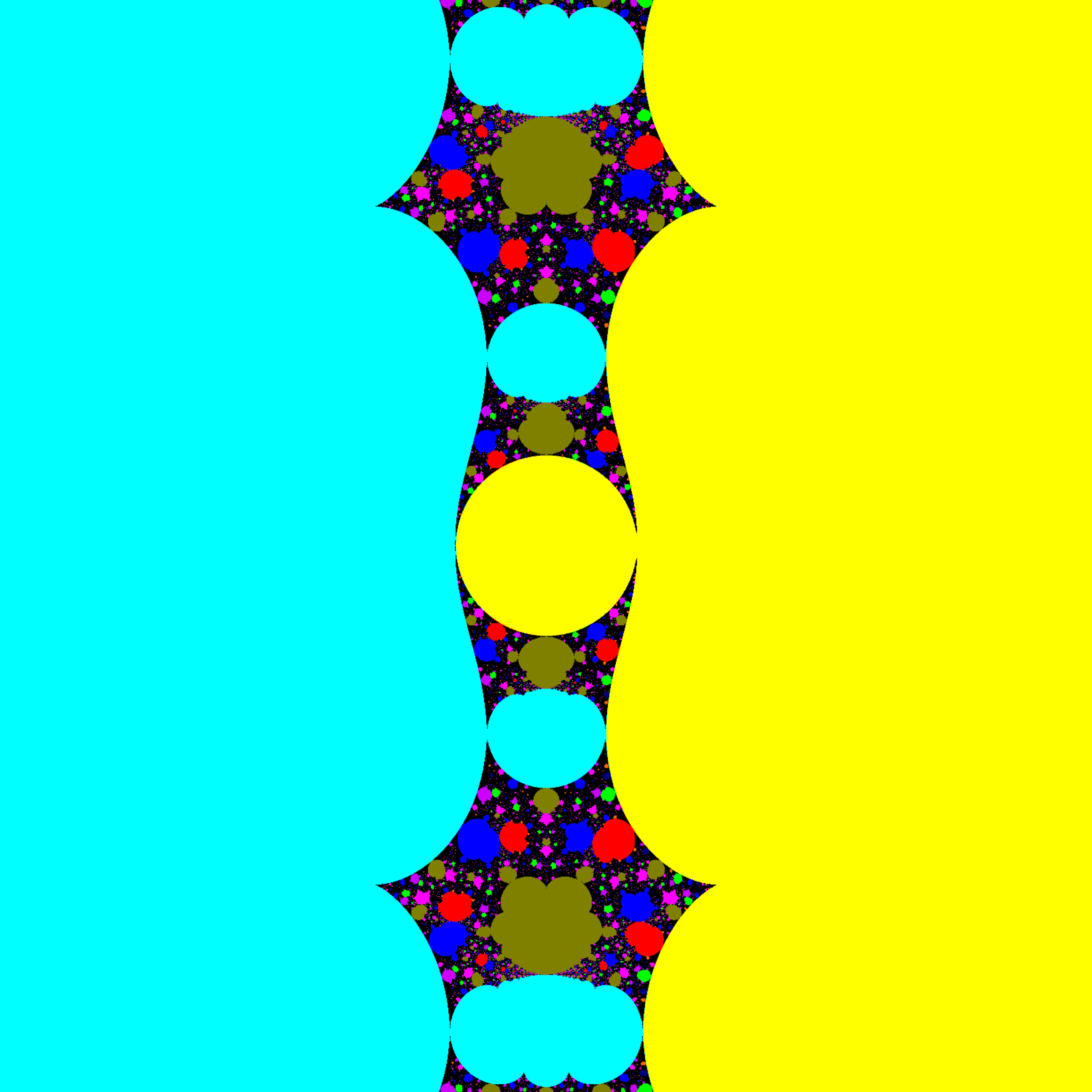}
\setlength{\unitlength}{0.5\textwidth}
	\put(-0.5,0.5){\circle*{0.01}}
	\put(-0.49,0.49){\scriptsize $0$}
\caption{\label{fig:tan} \small The parameter plane of the family $T_\la(z)=\la \tan z$. Dimensions $[-6,6]\times[-6,6]$.}
\end{figure}

   The figure shows the $\lambda$ plane where we can observe these symmetries.    We only follow the orbit of one asymptotic value, $\lambda i$, and we color the components based on the period of the cycle it is attracted to.   The colors do not reflect the behavior of the orbit of the other asymptotic value.       Thus, we see the same color  for a $\lambda$ value for which there are two separate attracting periodic cycles of period $2k$ as we do when there is a single attracting periodic cycle of period $2k$ attracting both asymptotic values.    

There is a single non-simply connected component, the punctured unit disk,    for which both asymptotic values are attracted to the origin.     There are two unbounded components, one on the right (yellow) and one on the left (cyan-blue).  In the one on the right, $\lambda i$ is attracted to a fixed point  $a_0$ with multiplier $\rho_0$,  and $-\lambda i$ is attracted to $-a_0$ with the same multiplier.  It has the same color as the unit disk.  In the unbounded component on the left, both $\lambda i$ and $-\lambda i$ are attracted to the attracting period two cycle $a_0, -a_0$ and its multiplier is $\rho_0^2$.     Neither of these components has a finite virtual center.  Although the left one is colored for period $2$,  it cannot  have a finite virtual center because if $\lambda^*$ were a finite virtual center, $-\lambda^*$ would  be a virtual center for the unbounded component on the right of period $1$, contradicting Corollary~\ref{cor:uniquevc}.

 Note that, except for the punctured disk, each bounded component is paired with another bounded component;  the two have a common virtual center and are tangent there.  The two unbounded components can be thought of as paired at infinity.   The relationship between these component pairs is discussed in   in \cite{KK} and in detail for $\lambda$ on the imaginary axis in \cite{CJKBif}. 
  
\end{enumerate}
\end{example1}
\begin{example1}
Functions in $\calf_2$ can be composed with rational or polynomial functions and, up to affine conjugation, the dynamics of the composed functions remain invariant under  quasi-conformal deformations supported on the Fatou set (see Theorem \ref{thm:function families} in the Appendix).  Below we describe  dynamically natural slices formed by pre- and post- composition with a quadratic polynomial $Q(z)=az^2 +bz +c$.  
\begin{enumerate}[{\rm (a)}]
\item\label{ex 2a'} 
 If we start with an an $ f \in \calf_2$ with finite asymptotic values and pre-compose by a degree two polynomial, $Q(z)=az^2 +bz +c$,  the resulting function
\[
g(z)=f(az^2 +bz +c)
\]
has  the same two asymptotic values as $f$ does.   Since   each  asymptotic tract of $f$ has two pre-images under $Q$, each asymptotic value of $g$ has two asymptotic tracts and so has multiplicity two.   Since $f$ has no critical points, $g$ has a single critical point at $-b/2a$ and a single critical value at $f(-b/2a)$.
If we assume the origin is fixed, then $c=0$.    If we assume it is also a critical value then $b=0$.   Now making one of the asymptotic values land on the fixed point $0$, we can form a dynamically natural slice.   This determines the coefficient $a$. This is a slice  in the relaxed sense of Remark \ref{relax},  because one of the singular values is preperiodic. 
The functions in the slice can then be written 
\[
g_{\la}(z)=\frac{e^{z^2} - e^{-z^2}}{\frac{1}{\la} e^{z^2} + (\frac{1}{\sqrt{\pi i}})e^{-z^2}}
\]
The asymptotic values are $\lambda$ and $-\sqrt{\pi i}$.  The singular parameter values are at $0$ and $-\sqrt{\pi i}$.
Theorem~\ref{thm:function families} implies that condition (d) holds. 

Figure~\ref{fig:BigpuremeroPiSq} shows the $\la$-plane.  The green components are capture components where the orbit of $\la$ falls into the basin of attraction of the origin. In the unbounded yellow shell components, $\la$ is attracted to an attracting (not super attracting) fixed point.   In the cyan-blue (bounded) components it is attracted to a cycle of period two, in the red ones to a cycle of period three and in the brownish green ones to a cycle of period four.  The components come in pairs because the asymptotic values have multiplicity two.   This is discussed  in \cite{CK19}.
\begin{figure}[htb!]
\centering
\captionsetup{width=0.85\textwidth}
\includegraphics[width=0.4\textwidth]{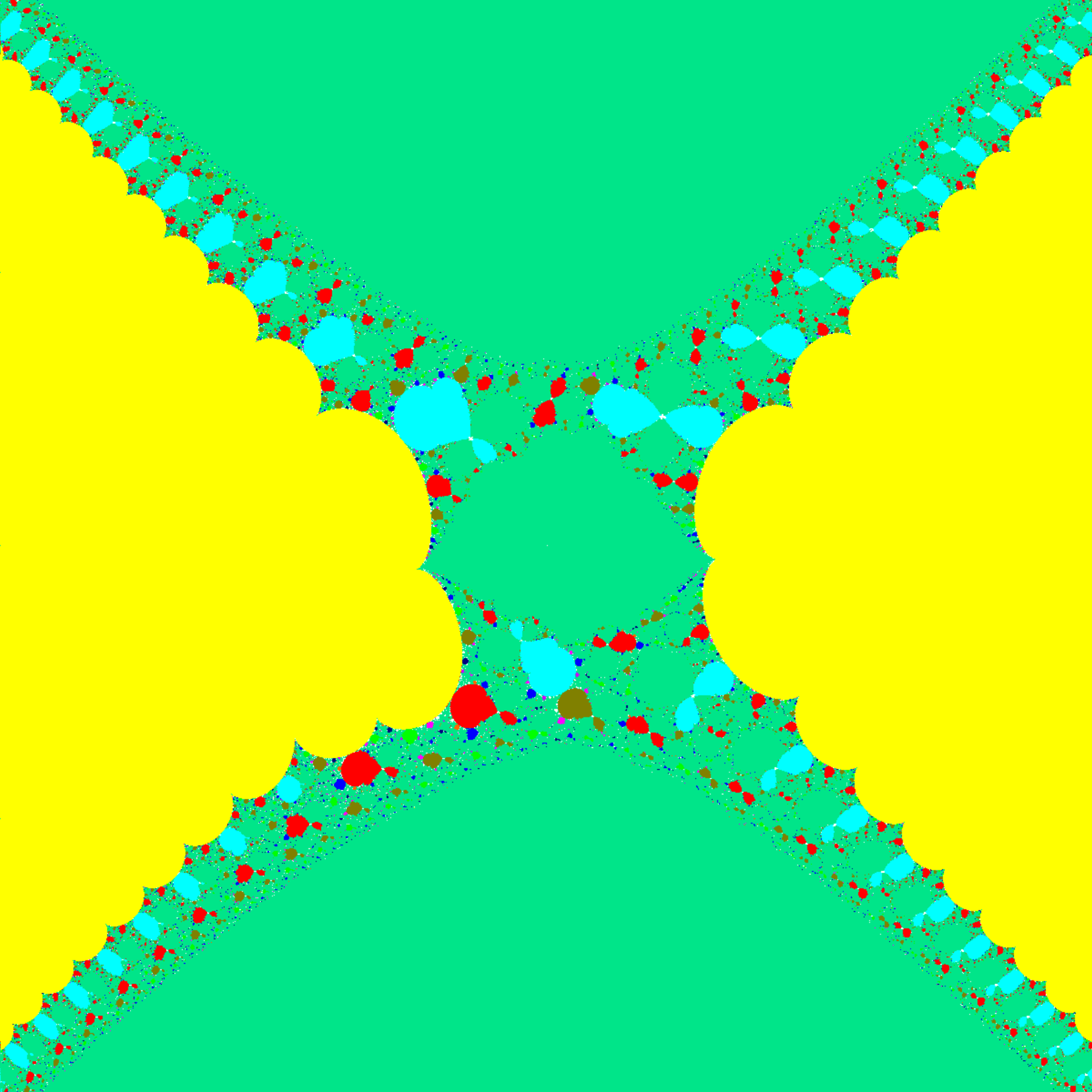}
\caption{\label{fig:BigpuremeroPiSq}  \small Dynamically natural slice for functions in $\calf_2$ pre-composed with a quadratic polynomial.}
\end{figure}


\item\label{ex 2c'}

We can form a  dynamically natural slice, in the relaxed sense, with the functions $h=af^2 + bf+c$,  for $f \in \calf_2$ by  assuming the origin is a super attracting fixed point and that the two asymptotic values coincide.  In this case  the functions in the slice each have one asymptotic value of multiplicity two.  We can write these functions as 
\[
\lambda \tanh^2 z.
\]
   The parameter $\lambda$ is the free (double) asymptotic value and so determines a dynamically natural slice of the parameter space $\Lambda = \C \setminus \{0\}$.   Condition (e) is satisfied in the modified sense.

In the left plot of Figure~\ref{fig:atan2hz} we see the $\lambda-$plane.  
In the center of the figure we see a green  capture component formed by parameters for which the asymptotic value belongs to the immediate basin of $0$.   This component is doubly  connected because of the puncture at $\la=0$. The remaining green components are   capture components  for higher iterates.  

   In the two yellow  shell components the  period is  one.  They are unbounded in accordance with  Proposition~\ref{period1}.  The remaining shell components are all of higher period and numerics indicate that they are all bounded. We see that they are grouped into quadruples that share a virtual center. This is because the asymptotic value has two asymptotic tracts and each asymptotic tract has two pre-asymptotic tracts.   
 
\begin{figure}[htb!]
\centering
\captionsetup{width=0.85\textwidth}
\includegraphics[width=0.45\textwidth]{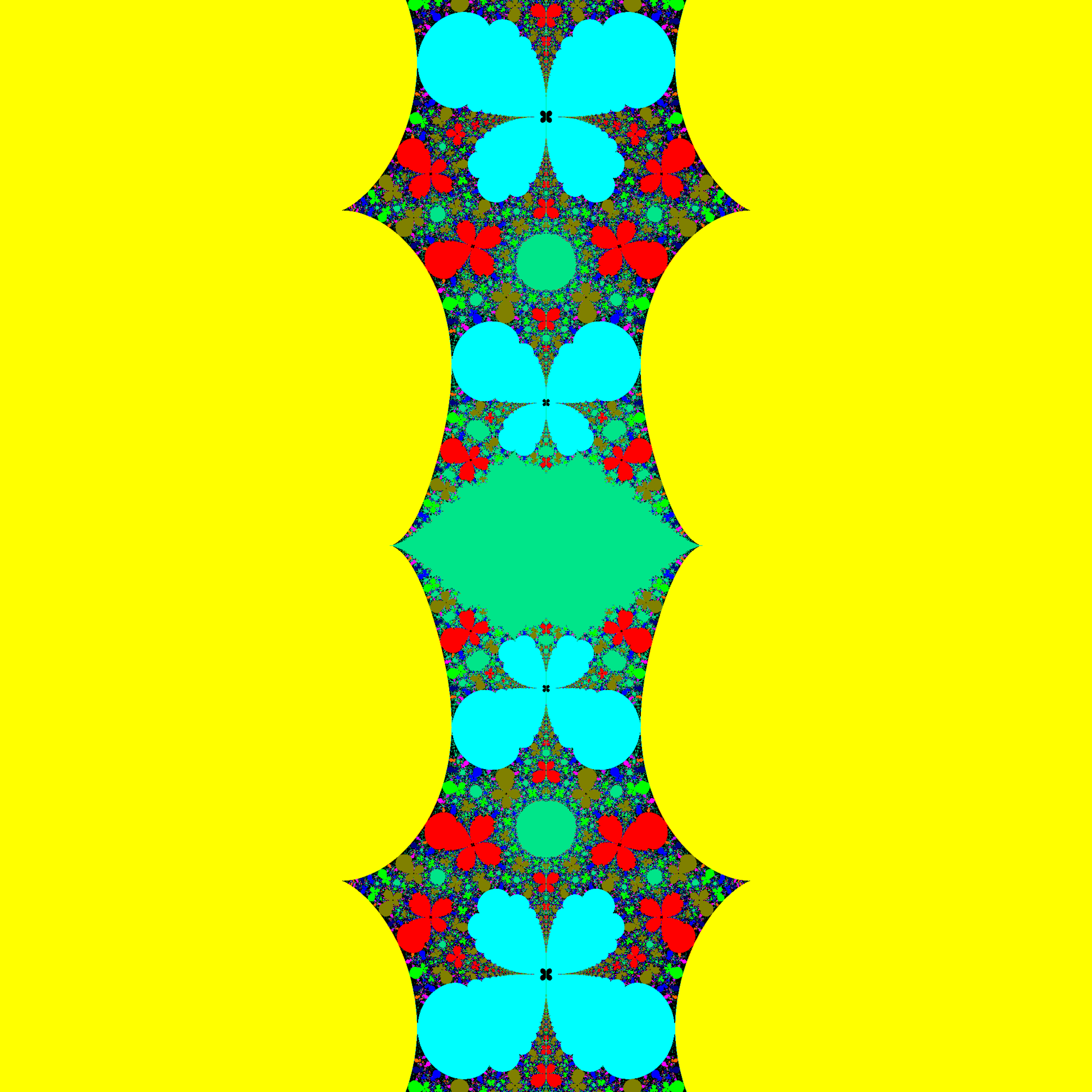} \hfil
\includegraphics[width=0.45\textwidth]{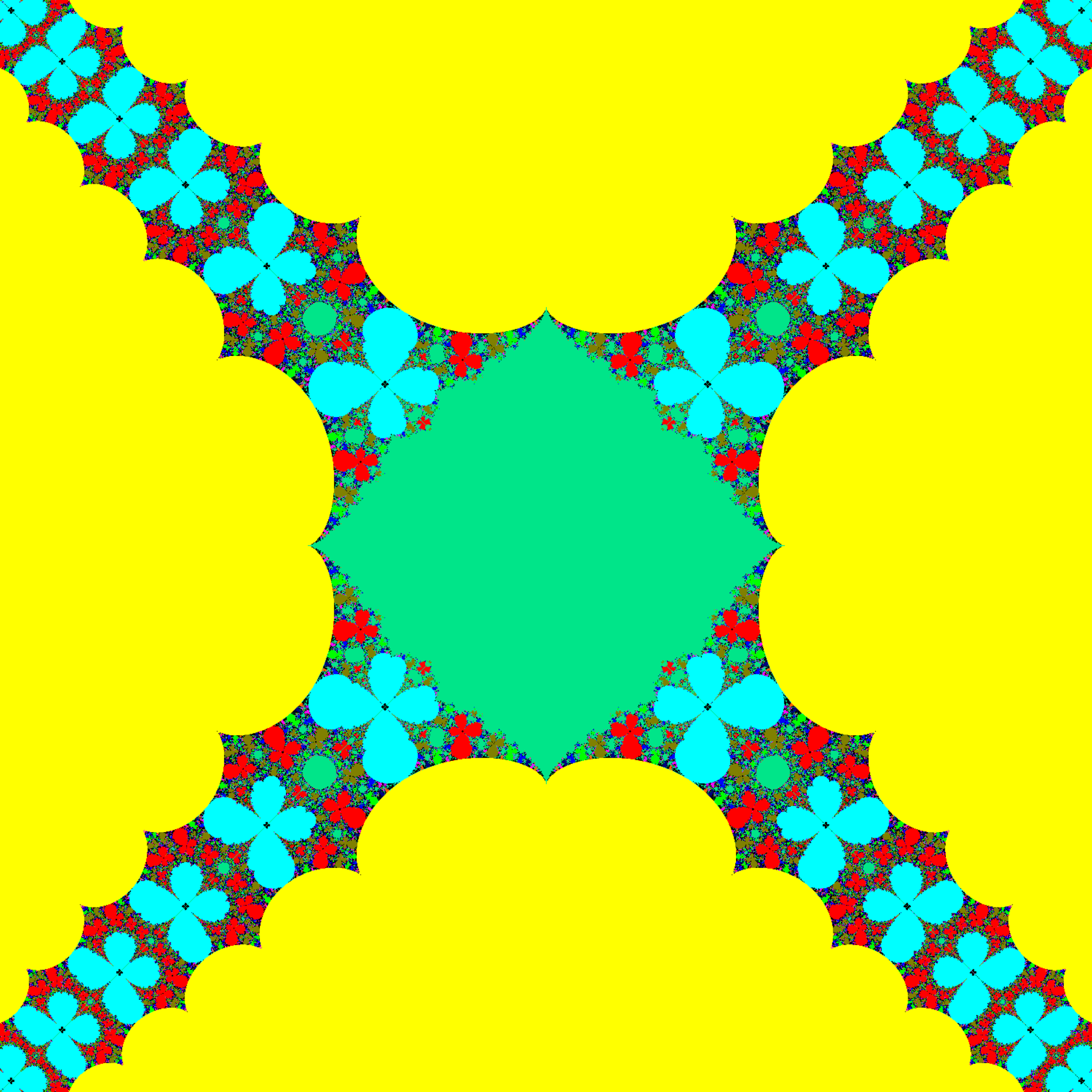}
\caption{\label{fig:atan2hz} \small Left: Parameter plane of the family $\lambda \tanh^2 z$. Dimensions $[-6,6]\times [-6,6]$. Right: Parameter plane of the family $\lambda \tanh z^2$.  Dimensions $[-3,3]\times [-3,3]$}
\end{figure}

Compare this figure to the right plot in Figure~\ref{fig:atan2hz}  which shows the parameter plane for the family $\lambda \tanh z^2$.   It  is essentially the ``square root" of this slice, since both maps are semiconjugate by $z^2$.  The asymptotic values now have multiplicity $2$.  These examples are investigated further in \cite{CK19} where it is proved that all shell components of period greater than $1$ are bounded.

 \end{enumerate}
\end{example1}
\begin{example1}\label{hybrid}
Our last example is a slice in the family $R(z) e^{Q(z)}$ (see the Appendix) where $R(z)$ is rational and $Q(z)$ is a polynomial.   Choosing the degrees of $R$ and $Q$ to be $1$, we obtain a function with two asymptotic values, one of which is infinity.  It also has two critical points and two critical values.   For our slice,  we fix one of the critical points at infinity so it is both a critical point and an asymptotic value.   We let   the second asymptotic value $a$ be free.   We normalize so that the origin is an attracting fixed point with fixed multiplier  $|\rho|<1$; the origin attracts the critical value that is the image of the finite critical point which we take as $\rho/(\rho+a)$.   Note there is then a single pole at $-a/(\rho+a)$.   With this normalization, the functions depend only on $a$ and have the form
\[  f_a(z) = a \big( 1-  \frac{a e^z}{(\rho+a)z+a} \big).    \]

\begin{figure}[htb!]
\centering
\captionsetup{width=0.85\textwidth}
\includegraphics[width=0.45\textwidth]{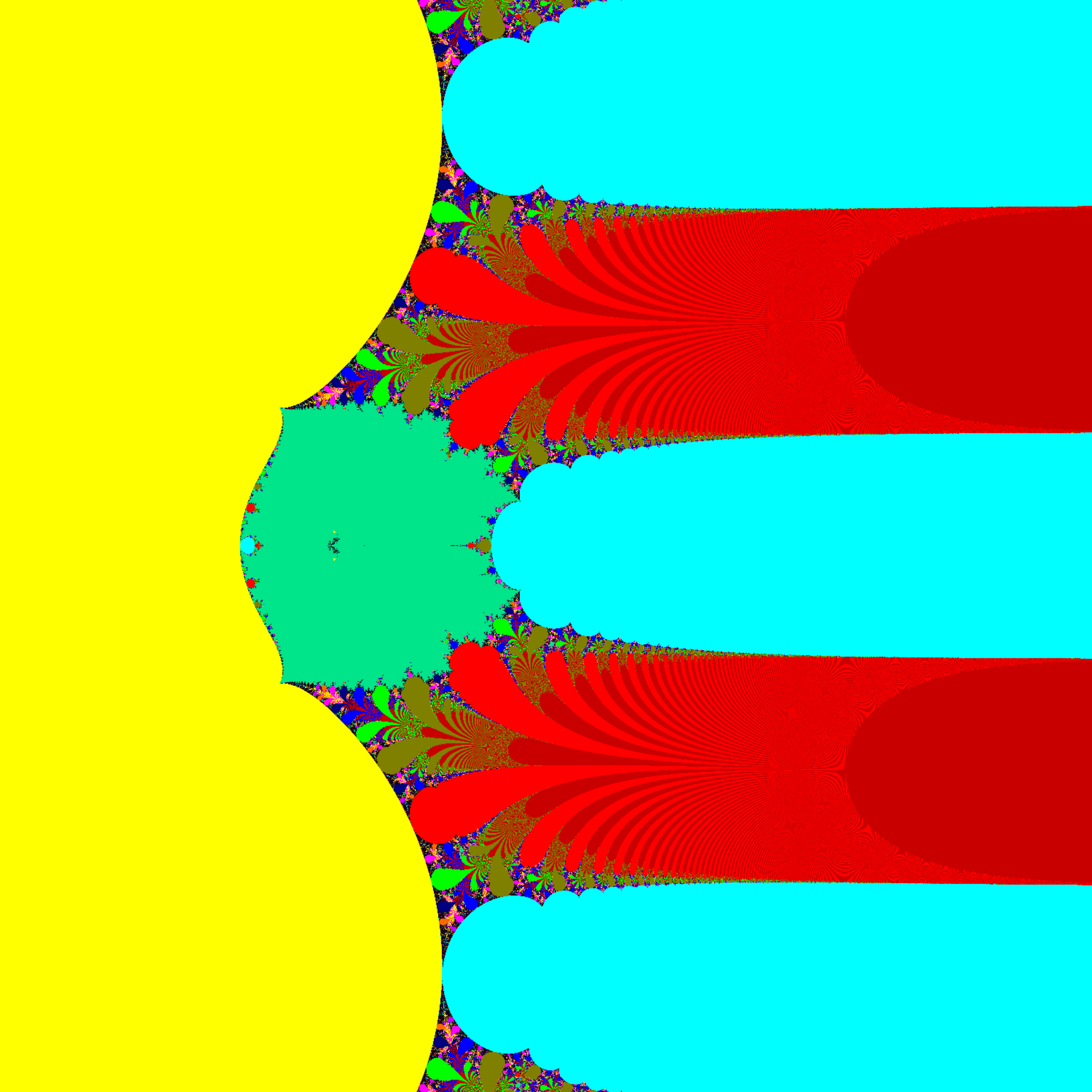} \hfil
\includegraphics[width=0.45\textwidth]{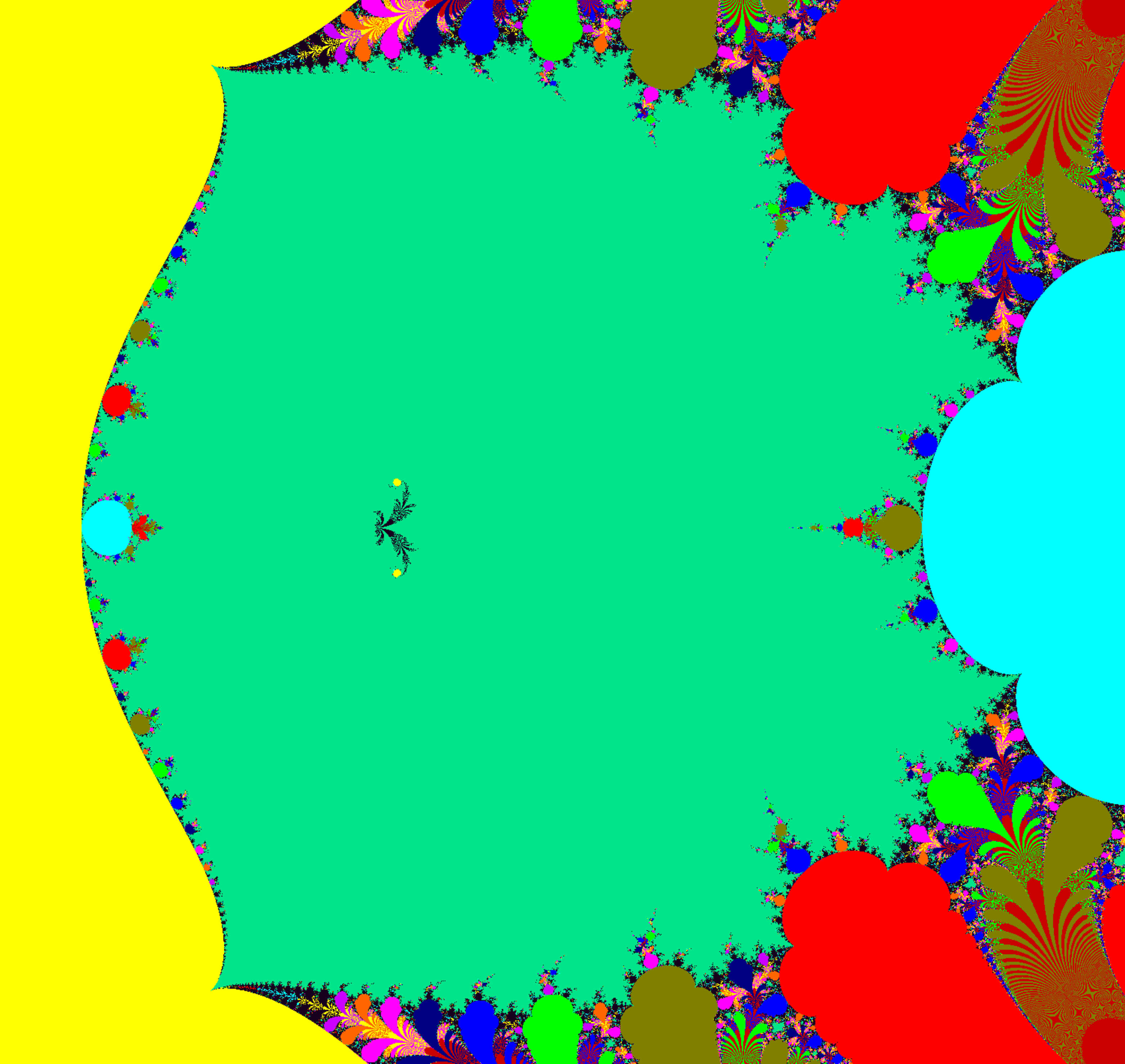}
\caption{\label{fig:newexamp} \small Left: Parameter plane of the family $f_a(z)$ with $\rho=1/2$. Dimensions $[-6,6]\times [-6,6]$. Right: Blow-up showing the Mandelbrot sets (tiny yellow in the center).  Dimensions $[-3,3]\times [-3,3]$}
\end{figure}

The previous examples contained only two types of components, shell components and capture components.  This example also contains very small yellow Mandelbrot sets at the ends of the Cantor bouquets inside the big green capture component in the blow-up. 
  They correspond to parameters where the critical point behaves independently, while the asymptotic value  $a$ is attracted to 0.  There are also   both bounded and unbounded shell components in the picture.   In the blow-up,  we see bounded shell components bifurcating from the unbounded yellow fixed one on the left.   The virtual center of the period two component on the right is at infinity.  

 This example has only one pole, so infinity must be an asymptotic value. It could be thought as a "hybrid" between a slice in an entire family and a meromorphic slice in $\calm_\infty$.     The other examples have either infinitely many poles or none.   
\end{example1}

\subsection*{Further examples} 
Other one-dimensional slices recording the position of a free asymptotic value while keeping the other ones under control have been studied in the literature. These are mostly entire functions and we refer the reader to  \cite{berfag,deniz15} for these examples.

\section{Appendix}
Any meromorphic function that is a   branched cover of the sphere of finite degree $d$ is a rational function and so can be expressed as a quotient of relatively prime polynomials at least one of which has degree $d$.   The $2d+1$ coefficients of these polynomials define a natural embedding of the space of such functions  into $\C^{2d+1}$.  Up to affine conjugation then,  these functions are represented as a complex analytic manifold of dimension $2d-2$.   This family is denoted by $Rat_d$ in the literature.

Not all holomorphic families of meromorphic functions have such an obvious representation as a complex manifold.   In this appendix we consider some transcendental families that do.   These are the families that most of our examples are drawn from.   The functions have explicit expressions that  involve complex constants.  We show that under quasiconformal deformation, the deformed function has a similar expression in which only the constants are changed.  Thus the constants determine the embedding of the manifold into $\C^n$ for the appropriate $n$.  

  We begin with families of 
meromorphic functions with $p< \infty$ asymptotic values, no critical points  and a single essential singularity at infinity.  We denote these by $\calf_p$.  (See \cite{DK, KK,EreGab} for further discussion.)  

 These functions have a particularly nice characterization.  Recall that the Schwarzian derivative of a function $g$ is defined by 
 \begin{equation}\label{schw}
 S(g)(z)=(g''/g')' - \frac{1}{2}(g''/g')^{2}.
 \end{equation}
     Nevanlinna, \cite{Nev, Nev1}, Chap X1, proved
\begin{theorem}  \label{thm:Nev}
Every meromorphic function $g$ with   $p < \infty$ asymptotic values and $q < \infty$ critical points  has the property that its Schwarzian derivative is a rational function of degree $p +q -2$.  
If $q=0$, the Schwarzian derivative is a polynomial $P(z)$.  
In the opposite direction,  for every polynomial function $P(z)$ of degree $ p-2$,  the solution to
the Schwarzian differential equation $S(g)=P(z)$ is a meromorphic function with exactly $p$ asymptotic values and no critical points.
\end{theorem}

 Since the theorem is classical and the proof is not well known, we sketch the proof  here and refer the  reader to the literature for details. 
\begin{proof} The   proof   follows from the construction of  a function with $p$ asymptotic values and no critical points as a limit of   rational functions with $p$ branch points.   Letting the   order of the branching at  some or all of these $p$ points increase, one obtains a sequence of  rational functions.  In the limit, the images of the branch points whose order goes to infinity become logarithmic singularities and
the limit function has finitely many branch points and finitely many logarithmic singularities.  The limit function is a parabolic covering map from the plane to itself.    The Schwarzian derivatives of the rational functions in the sequence are again rational functions with degree determined only by the number of branch points, not their order.  The limit of the Schwarzian derivatives is the Schwarzian derivative of the limit and so must be rational.

If all the branch points become asymptotic values in the limit, the limit function has  no critical points or critical values and hence its derivative never vanishes.  It follows from the definition of the Schwarzian that in the limit, it must be a polynomial. 
\end{proof}
 
 It is classical,  ( see e.g. \cite{hil}), that solutions to the  Schwarzian equation $S(g)=P(z)$  are related to solutions of the  linear second degree ordinary differential equation  
 \begin{equation} \label{ricatti}   
  w'' + \frac{1}{2} P(z) w = 0. 
  \end{equation}
Such an equation has a two dimensional space of holomorphic solutions.  If $w_{1}, w_{2}$ are a pair of linearly independent solutions, and 
\[ 
g= \frac{aw_{2} + bw_{1}}{cw_{2}+dw_{2}},  \,\,  a,b,c,d \in \C,  \, ad-bc = 1,
\] 
it is easy to check that $S(g)=P(z)$.   Moreover,  if $g$ is any solution  of the Schwarzian equation,  $w=\sqrt{1/g'}$ is a solution of the linear equation.  

 \begin{remark}\label{rem:cocycle}
One of the basic features of the Schwarzian derivative is that it satisfies the following cocycle relation:
if $f,g$ are meromorphic functions then
\begin{equation}\label{cocyle}
S(g ( f))(z) = S(g(f)) f'(z)^2 + S(f(z)).
\end{equation}
In particular, if $T$ is a M\"obius transformation,
$S(T(z))=0$ and $S(T\circ g(z) )= S((g(z))$ so that post-composing by $T$ does not change the Schwarzian.  Under pre-composition by a M\"obius transformation  the Schwarzian behaves like a quadratic differential.  In particular, pre-composing by an affine transformation multiplies the Schwarzian by a constant.  
\end{remark} 
This means that if we find a specific pair $(w_1^*,w_2^*)$ of solutions to the second order linear equation and set $g^* = w_2^*/w_1^*$, then every solution of the Schwarzian equation $S(g) = P(z)$ has the formula
\[ 
g= \frac{ag^* + b }{cg^*+d },  \,\,  a,b,c,d \in \C,  \, ad-bc = 1,
\] 
In particular, if $p=2$, $P(z)$ is identically constant.   Using one of the constants in an affine conjugation, we may assume $P \equiv -1/2$;  then  a specific  pair of  solutions to equation~(\ref{ricatti}) is $w_1^*=e^{-\frac{z}{2}}, w_2^*=e^{\frac{z}{2}}$ so we have $g^*=w_2^*/w_1^*=e^{z}$.   
The functions in the family $\calf_2$   have the form 
 \[ 
g(z)= \frac{ae^{z} + b }{ce^{z}+d },  \,\,  a,b,c,d \in \C,  \, ad-bc = 1.
\] 
 Since there is one more degree of freedom from the affine conjugation, we see this is a three dimensional family.   We have discussed several dynamically natural slices of $\calf_2$ in this paper.

For $\calf_3$, $P(z)$ is linear and  two specific solutions to the second order linear equation
\[
w''+\frac{1}{2} \zeta w = 0
\]
are given by the Airy functions\footnote{ These are named after the British astronomer G.B.Airy (1801-92).  Others have studied this equation and   solutions are expressed in terms of  Bessel functions and Gamma functions.  We won't write the formulas but will describe the properties we need.} $Ai(\zeta), Bi(\zeta)$.    Setting $g^*(\zeta)=Ai(\zeta)/Bi(\zeta)$  we obtain a solution with three asymptotic values,   $0, i$ and $ -i$.  Since we are interested in the dynamics, we may conjugate by the affine transformation $\zeta=rz+s$.   We may thus transform any function whose Schwarzian is linear, and hence any function in $\calf_3$, up to affine conjugation, to one given by the formula 
\[ 
g(\zeta)= \frac{a g^*(\zeta) + b }{cg^*(\zeta)+d },  \,\,  a,b,c,d \in \C,  \, ad-bc = 1.
\] 

The asymptotic values of $g$ are $\frac{b}{d},  \frac{ai +b }{ci+d}, \frac{-ai +b }{-ci+d }$.  

We could define a dynamically natural slice, for example,  by choosing the coefficients $a,b$ and $d$ so that 
\[
g(1)=1,  \mbox{ and } g'( 1)= 1/2.
\]
It is not hard to compute that the remaining coefficient $c$ is an affine function of the  asymptotic value $\frac{b}{d}$.  
 
 There are standard functions that solve the second order equation when the coefficient polynomial is of degree two or three, and they give formulas for functions in $\calf_4$ and $\calf_5$.  

The following  families of holomorphic  functions have rational Schwarzian derivatives and are  determined by their topological covering properties.  We list them  here and refer the reader to the cited literature for further discussion of them. 
\begin{itemize}
\item  
 Functions of type $Re^Q$ for $R$ rational and  $Q$ polynomial.  See  \cite{CJK1, Zak} for discussions with $R$ polynomial. 
  
\item  Functions of type $\int_{z_0}^{z} P(t)e^Q(t) dt$ for $P,Q$ polynomials.  See \cite[Th. 6.2]{baker84} and 
\cite[Prop. 2]{taniguchi}). 
\end{itemize}

\subsubsection*{Compositions}

 The following theorem shows how we can find more families for which the  functions have explicit formulas.   (See \cite{CK19} for further discussion of these functions).

\begin{theorem} \label{thm:function families} 
 Let $f_0 \in \calf_p$
  \begin{enumerate} 
\item  Suppose $g_0 = f_0 \circ Q_0$ is a function such that $Q_0$ is a polynomial of degree $d$ and suppose that $g$ is a meromorphic function quasiconformally conjugate to $g_0$.   
 Then $g$ can be expressed as $g=f \circ Q $ for some function $f \in \calf_p$ and some  polynomial $Q$  of degree $d$.   
\item Suppose $R_0 \in Rat_d$ is rational of degree $d$ and $h_0 = R_0 \circ f_0$.  Then if $h(z)$ is quasiconformally conjugate to $h_0$,  $h$ can be expressed as $h=R \circ f$ for functions $R \in Rat_d$ and $f \in \calf_p$.   
\end{enumerate}
\end{theorem}

\begin{remark} The proof of part 1 of the theorem works if $Q_0$ is replaced by a rational function $R_0 \in Rat_d$ but then the  composed function $g_0$  
 has essential singularities at the poles of $R_0$.
\end{remark}
\begin{remark} Note that in part 1, $g = f \circ Q$ is a function with rational Schwarzian.  For each asymptotic value of $f$ of multiplicity $m$, $g$ has an asymptotic value of multiplicity $dm$; moreover $g$ has the same critical points as $Q$, namely $2d-2$ critical points counted with multiplicity, $d-1$ of which are at infinity.   In part 2, however, it is no longer true that the Schwarzian of the composed function $h=R \circ f$ is a rational function.  For example, it  may have infinitely many critical points.  \end{remark}
\begin{remark}  Observe that if $f\in\calf_p$ and $Q$ is a degree $d$ polynomial then  $Q\circ f$ and $f\circ Q$ are semiconjugate by a degree $d$ polynomial --  in fact by $Q$. 
\end{remark}

\begin{proof}
 The proof of both parts of the theorem is  essentially the same.  We therefore carry it out only for part 1.  
 
In part 1, let $ \phi^{\mu}$ be a quasiconformal homeomorphism with Beltrami coefficient $\mu$ such that 
\[
g_{\mu} = \phi^{\mu} \circ g_0  \circ (\phi^{\mu})^{-1} 
\]
 is meromorphic.   We can use $f_0$ to pull back the complex structure defined by $\mu:= \bar{\partial}\phi^\mu / \partial \phi^\mu$ to obtain a complex structure $\nu=f_0^*\mu$ such that the map 
\[f_{\mu} =  \phi^{\mu} \circ f_0 \circ (\phi^{\nu})^{-1} \]
is meromorphic.   Note that this is not a conjugacy since it involves two different homeomorphisms. 

We can now write 
\[
g_{\mu} = \phi^{\mu} \circ f_0 \circ (\phi^{\nu})^{-1} \circ \phi^{\nu} \circ Q_0 \circ (\phi^{\mu})^{-1}
\]
and set 
\[ 
Q_{\mu}=\phi^{\nu} \circ Q_0 \circ (\phi^{\mu})^{-1}.
\]
Again this is not a conjugacy but it is meromorphic since $g_{\mu}$ and $f_{\mu}$ are homeomorphisms.

The main point here is that although $f_{\mu}$ is not   a conjugate of $f_0$, since the quasiconformal maps $\phi^{\mu}$ and  $\phi^{\nu}$   are homeomorphisms,  the map $f_{\mu}$    is a meromorphic  map with the same topology as $f_0$; that is, it has $p$ asymptotic values and no critical values.   By Nevanlinna's theorem, Theorem~\ref{thm:Nev}, $f_{\mu}$ belongs to $\calf_p$.   Similarly,  although $Q_{\mu}$ is not defined as a conjugate of $Q_0$,   since the quasiconformal maps $\phi^{\nu}$ and  $\phi^{\eta}$   are homeomorphisms,   the map $Q_{\mu}$ is a meromorphic map with the same topology as $Q_0$; that is, it is a degree $d$ branched covering of the Riemann sphere with the same number of critical points and the same branching as $Q_0$ and thus it must be a polynomial of degree $d$. 
  \end{proof}

%
%

\bibliography{bibnuria.bib}

\end{document}